\documentclass{amsart}
\usepackage[usenames,dvipsnames]{xcolor}
\usepackage[mathscr]{euscript}
\usepackage{comment}
\usepackage{lineno}
\usepackage{url}
\usepackage{hyperref}
\usepackage{subcaption, wrapfig}
\usepackage{amsmath, amsfonts, amssymb, amsthm, geometry, graphics, graphicx, multicol, multirow, enumerate, float, floatflt, fullpage, tikz, outlines, url, tabularx}
\usepackage{placeins}
\usepackage[utf8]{inputenc}
\usetikzlibrary{shapes.multipart, matrix, positioning, arrows,cd, decorations.pathmorphing, decorations.pathreplacing}
\usepackage{wasysym} 
\usepackage{mathdots} 
\usepackage{scalerel} 
\newlength\bshft 
\def\fakebold#1{\ThisStyle{\ooalign{$\SavedStyle#1$\cr%
  \kern-\bshft$\SavedStyle#1$\cr%
  \kern\bshft$\SavedStyle#1$}}}
  
\newtheorem{thm}{Theorem}[section]
\newtheorem{cor}[thm]{Corollary}

\newtheorem{lem}[thm]{Lemma}
\newtheorem{conj}[thm]{Conjecture}

\newtheorem*{theorem*}{Theorem \ref{TorusLinksOnly2Torsion}}
\newtheorem*{theorem'}{Theorem \ref{nooddtorsion}}
\newtheorem*{theorem"}{Theorem \ref{GeneralTheorem2}}

\theoremstyle{definition}
\newtheorem{defn}[thm]{Definition}
\newtheorem{exa}[thm]{Example}

\newcommand\restr[2]{{
  \left.\kern-\nulldelimiterspace 
  #1 
  \vphantom{\big|} 
  \right|_{#2} 
  }}
\newcommand{\R}{\mathbb{R}}
\newcommand{\Z}{\mathbb{Z}}
\newcommand{\Q}{\mathbb{Q}}
\newcommand{\F}{\mathbb{F}}

\newcommand{\zp}{\mathbb{Z}_p}

\newcommand{\rk}{\textnormal{rk }}

\newcommand{\im}{\textnormal{im}}

\newcommand{\red}{\text{red} \hspace{.5mm}}

\newcommand{\BN}{\textnormal{BN}}

\newcolumntype{P}[1]{>{\centering\arraybackslash}p{#1}} 
\definecolor{forest}{rgb}{0.03, 0.47, 0.19}

\graphicspath{ {./_figures/} }

\newcommand{\ACcom}[1]{{\textcolor{forest}{{\sf AC:} #1}}}

\newcommand{\AL}[1]{{\textcolor{purple}{{\sf AL:} #1}}}

\makeatletter
\def\Ddots{\mathinner{\mkern1mu\raise\p@
\vbox{\kern7\p@\hbox{.}}\mkern2mu
\raise4\p@\hbox{.}\mkern2mu\raise7\p@\hbox{.}\mkern1mu}}
\makeatother

\newsavebox\pinlinesigmaoneinverse
\begin{lrbox}{\pinlinesigmaoneinverse}

\end{lrbox}

\title{Torsion in thin regions of Khovanov homology}

\thanks{AC was supported in part by the project P31705 of the Austrian Science Fund. AL was supported in part by NSF grant DMS-1811344. RS partially supported by the Simons Foundation Collaboration Grant 318086 and NSF Grant DMS 1854705.}

\author{Alex Chandler}
\address{Faculty of Mathematics\\
University of Vienna\\
Vienna, Austria}
\email{alex.chandler@univie.ac.at}

\author{Adam M. Lowrance}
\address{Department of Mathematics\\
Vassar College\\
Poughkeepsie, NY} 
\email{adlowrance@vassar.edu}

\author{Radmila Sazdanovi\'{c}}
\address{Department of Mathematics\\
North Carolina State University\\
Raleigh, NC}
\email{rsazdanovic@math.ncsu.edu}

\author{Victor Summers}
\address{Division of Mathematics and Computer Science\\
University of South Carolina Upstate\\
Spartanburg, SC}
\email{vsummers@uscupstate.edu}

\begin{document}
\maketitle

\begin{abstract}
In the integral Khovanov homology of links, the presence of odd torsion is rare. Homologically thin links, that is links whose Khovanov homology is supported on two adjacent diagonals, are known to only contain $\Z_2$ torsion. In this paper, we prove a local version of this result. If the Khovanov homology of a link is supported on two adjacent diagonals over a range of homological gradings and the Khovanov homology satisfies some other mild restrictions, then the Khovanov homology of that link has only $\Z_2$ torsion over that range of homological gradings. These conditions are then shown to be met by an infinite family of $3$-braids, strictly containing all $3$-strand torus links, thus giving a partial answer to Sazdanovi\'{c} and Przytycki's conjecture that $3$-braids have only $\Z_2$ torsion in Khovanov homology. We use these computations and our main theorem to obtain the integral Khovanov homology for all links in this family.
\end{abstract}


\noindent\textbf{Mathematics Subject Classification:} 57K10, 57K18.

\section{Introduction}\label{intro}

In 1984, Jones discovered a powerful polynomial link invariant, now known as the Jones polynomial \cite{jones1997polynomial}. In 1999, Khovanov \cite{khovanov1999categorification} categorified the Jones polynomial $J(L)$ to a bigraded homology theory $H(L)$ called Khovanov homology. Khovanov homology categorifies the Jones polynomial in the sense that the Jones polynomial of a link can be recovered as the graded Euler characteristic of its Khovanov homology. 
Each $H^{i,j}(L)$ is a finitely generated abelian group, and the Jones polynomial gets a contribution only from the free part of $H(L)$. Thus torsion in Khovanov homology is a new phenomena in knot theory which does not appear in the theory of Jones polynomials. Throughout this paper we will use the term `$\Z_{p^r}$ torsion', $p$ a prime, to mean a direct summand of isomorphism class $\Z_{p^r}$ in the primary decomposition of the integral Khovanov homology of a link. 

Khovanov homology is equipped with two gradings: the homological grading $i$ and the polynomial grading $j$. A link is \textit{homologically thin} if its Khovanov homology is supported in bigradings $2i-j=s\pm 1$ for some integer $s$. Non-split alternating links and quasi-alternating links are homologically thin \cite{lee2008khovanov,manolescuozsvath2008quasi}. In \cite{shumakovitch2004torsion}, Shumakovitch showed that homologically thin links have only $\Z_{2^r}$ torsion in their Khovanov homology, and in \cite{shumakovitch2018torsion} he used a relationship between the Turner and Bockstein differentials on $\Z_2$-Khovanov homology to show in fact there is only $\Z_2$ torsion. In this paper, we prove a version of this result when the Khovanov homology of a link is thin over a restricted range of homological gradings.

Let $i_1,i_2\in\Z$ with $i_1<i_2$. We say that $H(L)$ is \textit{thin over $[i_1,i_2]$} if there is an $s\in\mathbb{Z}$ such that $H^{*,*}(L;\Z_p)$ is supported only in bigradings satisfying $2i-j=s\pm 1$ for all $i$ with $i_1\leq i \leq i_2$ and for all primes $p$. If $i$ is an integer with $i_1\leq i \leq i_2$, then we say $i\in[i_1,i_2]$; we similarly define $i\in(i_1,i_2], (i_1,i_2)$, or $[i_1,i_2)$.
\begin{theorem"}
Suppose that a link $L$ satisfies:
\begin{enumerate}
    \item $H(L)$ is thin over $[i_1,i_2]$ for integers $i_1$ and $i_2$ where $H^{[i_1,i_2]}(L)$ is supported in bigradings $(i,j)$ with $2i-j=s\pm 1$ for some $s\in\mathbb{Z}$,
    \item $\dim_{\Q}H^{i_1,*}(L;\Q)=\dim_{\Z_p} H^{i_1,*}(L;\Z_p)$ for each odd prime $p$,
    \item $H^{i_1,*}(L)$ is torsion-free, and
    \item $H^{i_1-1,j}(L)$ is trivial when $j\leq 2i_1-s -3.$
\end{enumerate}
Then all torsion in $H^{i,*}(L)$ is $\Z_2$ torsion for $i\in[i_1,i_2]$, that is, $H^{[i_1,i_2]}(L)\cong\Z^k\oplus \Z_2^\ell$ for some $k,\ell\geq 0$.
\end{theorem"}

We use Theorem \ref{GeneralTheorem2} to show that certain families of closed $3$-braids have only $\Z_2$ torsion in their Khovanov homology. Various techniques have been used to show that some other families of links only have $\Z_2$ torsion in their Khovanov homology or only have $\Z_2$ torsion in certain gradings. In \cite{helme2005torsion}, Helme-Guizon, Przytycki, and Rong established a connection between the Khovanov homology of a link and the chromatic graph homology of graphs associated to diagrams of the link. In \cite{lowrance2017chromatic}, Lowrance and Sazdanovi\'c used this connection to show that in a range of homological gradings, Khovanov homology contains only $\Z_2$ torsion. This result can now be seen as a corollary to Theorem \ref{GeneralTheorem2}.

 Przytycki and Sazdanovi\'c \cite{przytycki2012torsion} obtained explicit formulae for some torsion and proved that the Khovanov homology of semi-adequate links contains $\Z_2$ torsion if the corresponding Tait-type graph has a cycle of length at least 3. In the same paper the authors conjectured the following, connecting torsion in Khovanov homology to braid index: 
\begin{conj}[PS braid conjecture, 2012]
\label{psbraid}
\ \ 
\begin{enumerate}
\item The Khovanov homology of a closed 3-braid can have only $\Z_2$ torsion.
\item The Khovanov homology of a closed 4-braid cannot have $\Z_{p^r}$ torsion for $p \neq 2$.
\item The Khovanov homology of a closed 4-braid can have only $\Z_2$ and $\Z_4$ torsion.
\item The Khovanov homology of a closed $n$-braid cannot have $\Z_{p^r}$ torsion for $p>n$ ($p$ prime).
\item The Khovanov homology of a closed $n$-braid cannot have $\Z_{p^r}$ torsion for $p^r> n$.
\end{enumerate}
\end{conj}
\noindent Counterexamples to parts (2), (3) and (5) are given in \cite{mukherjee2017search}, and a counterexample to part (4) has recently been constructed by Mukherjee \cite{Mukherjee4} (see also \cite{sujoy2, sujoy}). However, part (1) remains open, and computations suggest that part (1) is indeed true. One goal of this (ongoing) project is to prove this.

\begin{figure}[ht]
\begin{center}
\begin{tikzpicture}
\node at (1,3){$\sigma_1 = \usebox\pinlinesigmaone$};
\node at (4.5,3){$\sigma_2 = \usebox\pinlinesigmatwo$};
\node at (6.45,3){$\dots$};
\node at (8.3,3){$\sigma_{s-1} = \usebox\pinlinesigmapminusone$};
\node at (.9,2){$\sigma_1^{-1} = \usebox\pinlinesigmaoneinverse$};
\node at (4.4,2){$\sigma_2^{-1} = \usebox\pinlinesigmatwoinverse$};
\node at (6.45,2){$\dots$};
\node at (8.3,2){$\sigma_{s-1}^{-1} = \usebox\pinlinesigmapminusoneinverse$};
\end{tikzpicture}
\end{center}
\caption{Generators and their inverses for the braid group $B_s$.}
\label{braidgroup}
\end{figure}
Consider the braid group $B_s$ on $s$-strands whose generators are shown in Figure \ref{braidgroup}.
By convention, braid words are read from left to right, and multiplication of words $w_1w_2$ corresponds to stacking the braid $w_2$ on top of $w_1$, see for example Figure \ref{braidclosure}. Each strand in a braid diagram is assumed to be oriented upward.
The \textit{half twist} $\Delta\in B_s$ is defined as $\Delta=(\sigma_1\sigma_2\dots\sigma_{p-1})(\sigma_1\sigma_2\dots\sigma_{p-2})\dots(\sigma_1\sigma_2)(\sigma_1)$, and the \textit{full twist} is $\Delta^2$. The closure of a braid diagram is a diagram of a link (see for example Figure \ref{braidclosure}), and a famous result of Alexander states that every link can be represented by the closure of a braid. For convenience, throughout this paper, a braid word will be used to refer to either an element of the braid group or its braid closure depending on the context in which it appears.
\begin{figure}[ht]
\centering
\begin{tikzpicture}
\node[scale=.8] at (0,0){\usebox\braidclosure};
\end{tikzpicture}
\caption{
        A braid diagram for the word $\Delta^2=\sigma_1\sigma_2\sigma_1\sigma_2\sigma_1\sigma_2\in B_3$ and its braid closure. }
\label{braidclosure}
\end{figure}
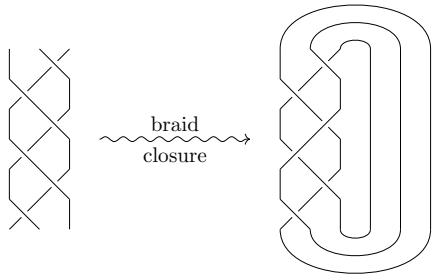
If two elements of a braid group are conjugate, then the corresponding braid closures are isotopic as links. Therefore, it would be convenient to have a classification of elements of the braid group $B_s$ up to conjugacy. For $s=2$, of course, the classification is trivial. For $s\geq 4$, no classification is known. For $s=3$, Murasugi provides the following \cite{murasugi1974closed}.

\begin{thm}[Murasugi]
\label{murasugi1}
Every element of the braid group $B_3$ is conjugate to a unique element of one of the following disjoint sets:
\begin{align*}
\Omega_0= & \; \{\Delta^{2n}\ | \ n\in\Z\},\\
\Omega_1=& \; \{\Delta^{2n}\sigma_1\sigma_2\ | \ n\in\Z\},\\
\Omega_2=& \;\{\Delta^{2n}(\sigma_1\sigma_2)^2\ | \ n\in\Z\},\\
\Omega_3=& \;\{\Delta^{2n+1}\ | \ n\in\Z\},\\
\Omega_4=& \;\{\Delta^{2n}\sigma_1^{-p}\ | \ n\in\Z\},\\
\Omega_5=& \;\{\Delta^{2n}\sigma_2^{q}\ | \ n\in\Z\},~\text{or}\\
\Omega_6=& \;\{\Delta^{2n}\sigma_1^{-p_1}\sigma_2^{q_1}\dots \sigma_1^{-p_r}\sigma_2^{q_r}\ | \ n\in\Z\},
\end{align*}
where $p, q, p_i$ and $q_i$ are positive integers. 
\end{thm}

With Murasugi's classification in mind, we use Theorem \ref{GeneralTheorem2} to show that certain classes of $3$-braids have only $\Z_2$ torsion in Khovanov homology (Theorem \ref{mainthm2}), taking a significant step in the direction of proving part (1) of the PS conjecture.

\begin{theorem*}
All torsion in the Khovanov homology of a closed 3-braid $L$ of type $\Omega_0, \Omega_1, \Omega_2,$ or $\Omega_3$ is $\Z_2$ torsion, that is,  $H(L)\cong\Z^k\oplus \Z_2^\ell$ for some $k,\ell\geq 0$.
\end{theorem*}

This paper is organized as follows. In Section \ref{Construction}, we give a construction of Khovanov homology. In Section \ref{ComputationalTools}, we provide several computational tools in the form of a long exact sequence and various spectral sequences associated to alternate differentials on the Khovanov complex. In Section \ref{TheMainResult}, we prove the main result of this paper, providing conditions on the Khovanov homology of a link $L$ under which all torsion in thin regions of the integral Khovanov homology of $L$ is $\Z_2$ torsion. In order to achieve this result, we will analyze interactions between the spectral sequences of Section \ref{ComputationalTools}. In Section \ref{application}, we give an application of the main result, showing that all torsion in the Khovanov homology of links in $\Omega_0, \Omega_1, \Omega_2$ and $\Omega_3$ is $\Z_2$ torsion, and we give explicit calculations of the integral Khovanov homology of links in $\Omega_0, \Omega_1, \Omega_2$ and $\Omega_3$. We end with an explanation of why Theorem \ref{GeneralTheorem2} does not apply to all closed $3$-braids.

\section*{Acknowledgements} The authors are thankful for helpful conversations with John McCleary and Alex Shumakovitch, and for the indispensable comments of the anonymous referees.

\section{A construction of Khovanov homology}\label{khovanovhomology}\label{Construction}

Khovanov homology is an invariant of oriented links $L \subset S^3$ with values in the category of bigraded modules over a commutative ring $R$ with identity. We begin with a construction of Khovanov homology. Our conventions for positive and negative crossings, and for zero- and one-smoothings, are as given in Figure \ref{Conventions}.
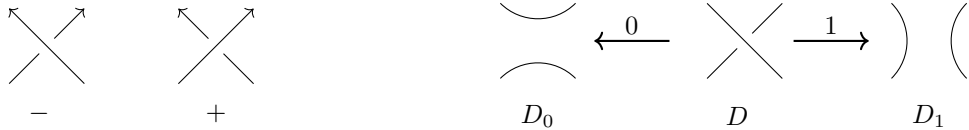
\begin{figure}[ht]
\centering
\begin{tikzpicture}
\useasboundingbox (-2,-1) rectangle (11.5,1);
\node at (0,0){\usebox\posnegconv};
\node at (8,0){\usebox\crossingconv};
\end{tikzpicture}
\caption{Conventions for crossings and smoothings in link diagrams.}
\label{Conventions} 
\end{figure}

\noindent A bigraded $R$-module is an $R$-module $M$ with a direct sum decomposition of the form $M = \bigoplus_{i,j \in \Z} M^{i,j}$. The submodule $M^{i,j}$ is said to have \textit{bigrading} $(i,j)$. For our purposes we will refer to $i$ as the \textit{homological grading} and to $j$ as the \textit{polynomial grading}. Given two bigraded $R$-modules $M = \bigoplus_{i,j \in \Z} M^{i,j}$ and $N = \bigoplus_{i,j \in \Z} N^{i,j}$, we define the direct sum $M \oplus_R N$ and tensor product $M \otimes_R N$ to be the bigraded $R$-modules with components $(M \oplus N)^{i,j} = M^{i,j} \oplus N^{i,j}$ and $(M \otimes N)^{i,j} = \bigoplus_{k+m=i,l+n=j} M^{k,l} \otimes N^{m,n}$. We also define homological and polynomial shift operators, denoted $[\cdot]$ and $\{ \cdot \}$, respectively, by $(M[r])^{i,j} = M^{i-r,j}$ and $(M\{s\})^{i,j} = M^{i,j-s}$.

Consider the directed graph whose vertex set $\mathcal{V}(n) = \{0,1\}^n$ comprises $n$-tuples of $0$'s and $1$'s and whose edge set $\mathcal{E}(n)$ contains a directed edge from $I \in \mathcal{V}(n)$ to $J \in \mathcal{V}(n)$ if and only if all entries of $I$ and $J$ are equal except for one, where $I$ is $0$
 and $J$ is $1$. One can think of the underlying graph as the 1-skeleton of the $n$-dimensional cube. For example, if $n=4$ there are two outward edges starting from $(0,1,0,1)$, one ending at $(0,1,1,1)$ and the other ending at $(1,1,0,1)$. The height of a vertex $I = (k_1,k_2,\ldots,k_n)$ is $h(I) = k_1 + k_2 + \cdots + k_n$. In other words, $h(I)$ is the number of $1$'s in $I$. If $\varepsilon$ is an edge from $I$ to $J$ and they differ in the $r$-th entry, the height of $\varepsilon$ is $|\varepsilon| := h(I)$. The sign of $\varepsilon$ is $(-1)^{\varepsilon} := (-1)^{\sum_{i=1}^{r-1} k_i}$. In other words, the sign of $\varepsilon$ is $-1$ if the number of $1$'s
before the $r$-th entry is odd, and is $+1$ if even.

Let $D$ be a diagram of a link $L$ with $n$ crossings $c_1, c_2, \ldots, c_n$. To each vertex $I \in \mathcal{V}(n)$ we associate a collection of circles $D(I)$, called a Kauffman state, obtained by $0$-smoothing those crossings $c_i$ for which $k_i=0$ and $1$-smoothing those crossings $c_j$ for which $k_j=1$. The $0$- and $1$-smoothing conventions are given in Figure \ref{Conventions}. To each Kauffman state we associate a bigraded $R$-module $C(D(I))$ as follows. Let $\mathcal{A}_2 = R[X]/X^2 \cong R1\oplus RX$ be the bigraded module with generator $1$ in bigrading $(0,1)$ and generator $X$ in bigrading $(0,-1)$. Denoting the number of circles in $D(I)$ by $|D(I)|$, we define 
\[C(D(I)) := \mathcal{A}_2^{\otimes |D(I)|}[h(I)]\{ h(I) \}.\]
Here, each tensor factor of $\mathcal{A}_2$ is understood to be associated with a particular circle of $D(I)$. Having associated modules with vertices, we now associate maps with directed edges. Define $R$-module homomorphisms $m : \mathcal{A}_2 \otimes \mathcal{A}_2 \rightarrow \mathcal{A}_2$ and $\Delta : \mathcal{A}_2 \rightarrow \mathcal{A}_2 \otimes \mathcal{A}_2$ (called multiplication and comultiplication, respectively) by
\[
	m(1 \otimes 1) = 1, \quad m(1 \otimes X) = X, \quad m(X \otimes 1) = X, \quad m(X \otimes X) = 0,
\]
\[
	\Delta(1) = 1 \otimes X + X \otimes 1, \quad \Delta(X) = X \otimes X.
\]
To an edge $\varepsilon$ from $I$ to $J$ we associate the map $d_{\varepsilon} : C(D(I)) \rightarrow C(D(J))$ defined as follows.
\begin{enumerate}[(1)]

\item If $|D(J)| = |D(I)| - 1$, then $D(J)$ is obtained by merging two circles of $D(I)$ into one: $d_{\varepsilon}$ acts as multiplication $m$ on the tensor factors associated to the circles being merged, and acts as the identity map on the remaining tensor factors.

\item If $|D(J)| = |D(I)| + 1$ then $D(J)$ is obtained by splitting one circle of $D(I)$ into two: $d_{\varepsilon}$ acts as comultiplication $\Delta$ on the tensor factor associated to the circle being split, and acts as the identity map on the remaining tensor factors.

\end{enumerate}
Suppose $D$ has $n_+$ positive crossings and $n_-$ negative crossings. Define the bigraded module 
\begin{equation}\label{KhovanovComplex}
C^{*,*}(D) = \bigoplus_{I \in \mathcal{V}(n)} C(D(I))[-n_-]\{ n_+ - 2n_- \}.
\end{equation}
Define maps $d^i : C^{i,*}(D) \rightarrow C^{i+1,*}(D)$ by $d^i = \sum_{|\varepsilon| = i} (-1)^{\varepsilon} d_{\varepsilon}$. In \cite{khovanov1999categorification}, Khovanov shows that $(C(D),d)$ is a (co)chain complex, and since the maps $d^i$ are polynomial degree-preserving, we get a bigraded homology $R$-module
\[
H(L;R) = \bigoplus_{i,j \in \Z} H^{i,j}(L;R)
\]
called the Khovanov homology of the link $L$ with coefficients in $R$, or, more compactly, the $R$-Khovanov homology of $L$. Proofs that $d$ is a differential and that these homology groups are independent of the diagram $D$ and the ordering of the crossings can be found in \cite{khovanov1999categorification,viro2004khovanov,bar2005khovanov}. If $R=\Z$ we simply write $H(L) = H(L;\Z)$. In this article we will focus on the rings $R = \Z, \Q$, and $\Z_p$, where $p$ is a prime.\\

\noindent The (unnormalized) Jones polynomial is recovered as the graded Euler characteristic of Khovanov homology:
\[
\hat{J}_L(q) = \sum_{i,j \in \Z} (-1)^i q^j \cdot \textnormal{rk}(H^{i,j}(L)).
\]

\section{Computational tools}\label{ComputationalTools}

\subsection{A long exact sequence in Khovanov homology}\label{computational}\label{LES}

Given an oriented link diagram $D$, there is a long exact sequence relating $D$ to the zero and one smoothings $D_0$ and $D_1$ at a given crossing of $D$ \cite{khovanov1999categorification}. Each crossing in an oriented link diagram is either positive $\left(~\tikz[baseline=.6ex, scale = .4]{
\draw[->] (0,0) -- (1,1);
\draw (1,0) -- (.7,.3);
\draw[->] (.3,.7) -- (0,1);
}~\right)$ or negative $\left(~\tikz[baseline=.6ex, scale = .4]{
\draw[->] (.7,.7) -- (1,1);
\draw[->] (1,0) -- (0,1);
\draw (0,0) -- (.3,.3);
}~\right)$.
If the crossing is negative, we set $c=n_-(D_0)-n_-(D)$ to be the number of negative crossings in $D_0$ minus the number in $D$. For each $j$ there is a long exact sequence

\begin{equation}
\begin{tikzcd}
\arrow[r,"\delta_*"]&H^{i,j+1}(D_1) \arrow[r] 
& H^{i,j}(D)\arrow[r]& H^{i-c,j-3c-1}(D_0)\arrow[r,"\delta_*"] & H^{i+1,j+1}(D_1)\arrow[r]&{}
\end{tikzcd}
\label{negativeLES}
\end{equation}

Note that in this case $D_1$ inherits an orientation, and $D_0$ must be given one. Similarly, if the crossing is positive, we set $c=n_-(D_1)-n_-(D)$ to be the number of negative crossings in $D_1$ minus the number in $D$. For  each $j$ there is a long exact sequence
\begin{equation}
\begin{tikzcd}
\arrow[r,"\delta_*"]&H^{i-c-1,j-3c-2}(D_1) \arrow[r] 
& H^{i,j}(D)\arrow[r]& H^{i,j-1}(D_0)\arrow[r,"\delta_*"] & H^{i-c,j-3c-2}(D_1)\arrow[r]&{}
\end{tikzcd}
\label{positiveLES}
\end{equation}
and note that in this case $D_0$ inherits an orientation, and $D_1$ must be given one.

\subsection{The Bockstein spectral sequence}\label{BocksteinSS}

The Bockstein spectral sequence arises as the spectral sequence associated to an exact couple, and is a powerful tool for analyzing torsion in homology theories.

\begin{defn}
Let $D_0$ and $E_0$ be $R$-modules, and let $i_0 : D_0 \rightarrow D_0$, $j_0 : D_0 \rightarrow E_0$ and $k_0 : E_0 \rightarrow D_0$ be $R$-module homomorphisms satisfying $\ker(i_0) = \im(k_0)$, $\ker(j_0) = \im(i_0)$ and $\ker(k_0) = \im(j_0)$. The pentuple $(D_0,E_0,i_0,j_0,k_0)$ is called an exact couple. Succinctly, an exact couple is an exact diagram of $R$-modules and homomorphisms of the following form.

\begin{center}
\begin{tikzcd}
D_0 \ar[rr,"i_0"] &  & D_0 \ar[ddl, "j_0"]\\
& &\\
& E_0 \ar[luu, "k_0"]&
\end{tikzcd}
\end{center}

\end{defn}

The map $d_0 : E_0 \rightarrow E_0$ defined by $d_0 := j_0 \circ k_0$ satisfies $d_0 \circ d_0 = (j_0 \circ k_0) \circ (j_0 \circ k_0) = j_0 \circ (k_0 \circ j_0) \circ k_0 = 0$ and so defines a differential on $E_0$. Define $E_1=H(E_0,d_0)$ to be the homology of the pair $(E_0,d_0)$, and define $D_1 = \im(i_0 : D_0 \rightarrow D_0)$. Also, define maps $i_1 : D_1 \rightarrow D_1$, $j_1 : D_1 \rightarrow E_1$ and $k_1 : E_1 \rightarrow D_1$ by $i_1 = i_0|_{D_1}$, $j_1(i_0(a)) = j_0(a) + d_0E_0$ and $k_1(e + d_0E_0) = k_0(e)$. The resulting pentuple $(D_1,E_1,i_1,j_1,k_1)$ is another exact couple \cite[Proposition 2.7]{mccleary2001user}. Iterating this process produces a sequence $(E_r,d_r)_{r \ge 0}$ called the \textit{spectral sequence associated to the exact couple} $(D_0,E_0,i_0,j_0,k_0)$.

Given a chain complex $(C,d)$, let us denote its integral homology by $H(C)$ and its homology over coefficients in $\Z_p$ by $H(C;\Z_p)$. Consider the short exact sequence
\begin{equation}\label{BocksteinSES1}
\begin{tikzcd}
0 \ar[r] & \Z \ar[r,"\times p"] & \Z \ar[r,"\red p"] & \Z_p \ar[r] & 0
\end{tikzcd}
\end{equation}
where $\times p$ is multiplication by $p$, and $\red p$ is reduction modulo $p$. Tensoring a chain complex $(C,d)$ with (\ref{BocksteinSES1}) yields a short exact sequence of chain complexes

\begin{center}
\begin{tikzcd}
0 \ar[r] & C \ar[r,"\times p"] & C \ar[r,"\red p"] & C \otimes \Z_p \ar[r] & 0.
\end{tikzcd}
\end{center}
The associated long exact sequence in homology can be viewed as an exact couple

\begin{equation}\label{BocksteinExactCouple}
\begin{tikzcd}
H(C) \ar[rr,"\times p"] &  & H(C) \ar[ddl, "\red p"]\\
& &\\
& H(C;\Z_p) \ar[luu, "\partial"]&
\end{tikzcd}
\end{equation}
where $\partial$ is the connecting homomorphism.

\begin{defn}
Let $p$ be prime. The spectral sequence $(E_B^r,d_B^r)_{r \ge 1}$ associated to the exact couple (\ref{BocksteinExactCouple}) is called the $\Z_p$ Bockstein spectral sequence.
\end{defn}

The following properties of the Bockstein spectral sequence will be of great importance for our purposes, the proofs of which can be found in \cite[Chapter 10]{mccleary2001user} and \cite[Proposition 3E.3]{hatcher2002algebraic}. 

\begin{enumerate}[(B1)]

\item The first page is $(E_B^1,d_B^1) \cong (H(C;\Z_p), \beta)$ where $\beta = (\red p) \circ \partial$. 

\item The infinity page is $E_B^{\infty} \cong (H(C)/\text{Torsion}) \otimes \Z_p$.

\item If the Bockstein spectral sequence collapses on the $E^k$ page for a particular bigrading $(i,j)$, that is, if $(E_B^k)^{i,j} = (E_B^{\infty})^{i,j}$, then $H^{i,j}(C)$ contains no $\Z_{p^r}$ torsion for $r \ge k$. 

\item If $\rk (d_B^1)^{i-1,j}$ is equal to the number of summands of $\Z_{p^r}$ in $H^{i,j}(C)$ for all $r$,
then there is no torsion of order $p^{\ell}$ in $H^{i,j}(C)$ for $\ell\geq 2$.
\label{bockprop}


\item If $C=C^{*,*}$ is a bigraded complex with a bidegree $(1,0)$ differential, then the differentials $d_B^r$ also have bidegree $(1,0)$.

\end{enumerate}

For our purposes, we use the $\Z_2$ Bockstein spectral sequence and property (B\ref{bockprop}) to prove there is only $\Z_2$ torsion in certain bigradings.
To this end, we will consider interactions between the Bockstein spectral sequence and the Turner spectral sequence, discussed below.

\subsection{Bar-Natan homology and the Turner spectral sequence}\label{specseq}\label{TurnerSS}

    In \cite{bar2005khovanov} Bar-Natan constructs another link invariant in the form of a bigraded homology theory $\BN^{*,*}(L)$, this time over the ring $\Z_2[u]$, where $u$ is a formal variable. 
    This construction mirrors that of Khovanov homology as in Section \ref{Construction}, but using a different Frobenius algebra. Setting $u=1$, one obtains the singly graded \textit{filtered Bar-Natan homology}, denoted $\BN^*(L)^\prime$. 
    Filtered Bar-Natan homology can be constructed as in Section \ref{Construction}, but with differential defined via the following multiplication and comultiplication:
\[
	m_B(1 \otimes 1) = 1, \quad m_B(1 \otimes X) = X, \quad m_B(X \otimes 1) = X, \quad m_B(X \otimes X) = X,
\]
\[
	\Delta_B(1) = 1 \otimes X + X \otimes 1 + 1 \otimes 1, \quad \Delta_B(X) = X \otimes X.
\]
%
    Turner shows the following in \cite[Theorem 3.1]{turner2006calculating}. 

\begin{lem}[Turner]\label{Bar-Natan dimension}
Let $L$ be an oriented link with $k$ components. Then, $\dim_{\Z_2}(\BN^*(L)^\prime) = 2^k$. Specifically, if the components of $L$ are $L_1, L_2, \ldots, L_k$, then
\[
	\dim_{\Z_2}(\BN^i(L)^\prime) = \#\left\{ E \subset \{1,2, \ldots, k \} \left| \ 2 \times \sum_{\ell \in E, m \in E^c} lk(L_\ell,L_m) = i \right. \right\}
\]
where $lk(L_\ell,L_m)$ is the linking number between $L_\ell$ and $L_m$.
\end{lem}

\begin{exa}
Let $T(3,q)$ be the $(3,q)$-torus link. Specifically, $T(3,3n)$ is the closure of the braid $\Delta^{2n}$, $T(3,3n+1)$ is the closure of the braid $\Delta^{2n}\sigma_1\sigma_2$ and $T(3,3n+2)$ is the closure of $\Delta^{2n}(\sigma_1\sigma_2)^2$, each oriented so that crossings in $\Delta$ are negative crossings.
\begin{enumerate}[(1)]\label{BN(T(3,3n))}

\item The torus link $T(3,3n)$ has three components and $BN^0(T(3,3n)) \cong \Z_2^2$, $\BN^{-4n}(T(3,3n))^\prime \cong \Z_2^6$ and $\BN^i(T(3,3n))^\prime\cong 0$ 
for $i\neq 0$ or $-4n$. 

\item The torus knot $T(3,3n+1)$ satisfies $\BN^0(T(3,3n+1))^\prime \cong \Z_2^2$ and $\BN^i(T(3,3n+1))^\prime \cong 0$ for $i\neq 0$.

\item The torus knot $T(3,3n+2)$ satisfies $\BN^0(T(3,3n+2))^\prime \cong \Z_2^2$ and $\BN^i(T(3,3n+2))^\prime \cong 0$ for $i\neq 0$.

\end{enumerate}
\label{exampleturner}
\end{exa}

In \cite{turner2006calculating} Turner defines a map $d_T$ on the $\Z_2$-Khovanov complex in the same manner as the Khovanov differential $d$, but with multiplication and comultiplication given by
\[
	m_T(1 \otimes 1) = 0, \quad m_T(1 \otimes X) = 0, \quad m_T(X \otimes 1) = 0, \quad m_T(X \otimes X) = X,
\]
\[
	\Delta_T(1) = 1 \otimes 1, \quad \Delta_T(X) = 0,
\]
and the map $d_T$ satisfies $d_T^2=0$ and $d \circ d_T + d_T \circ d = 0$.

\begin{defn}[Turner]\label{TurnerSSproperties}
Let $D$ be a diagram of a link $L$ and $(C(D;\Z_2),d)$ be the $\Z_2$-Khovanov complex. The spectral sequence $(E_T^r, d_T^r)_{r \ge 1}$ associated to the double complex $(C(D;\Z_2),d,d_T)$ is called the Turner spectral sequence.
\end{defn}

The Turner spectral sequence satisfies the following properties.

\begin{enumerate}[(T1)]
\item The first page is $(E_T^1,d_T^1) \cong (H(L;\Z_2),d_T^*)$ where $d_T^* : H(L;\Z_2) \rightarrow H(L;\Z_2)$ is the induced map on homology.

\item Each map $d_T^r$ is a differential of bidegree $(1,2r)$. \label{diffTurner}

\item  If $L$ is homologically thin over $\Z_2$, then the Turner spectral sequence collapses at the second page.
\label{Turnercollapse}

\item 
The dimension of the infinity page of the Turner spectral sequence is $2^n$ where $n$ is the number of components of $L$. There is a generator $s_{\mathfrak{o}}$ for each orientation $\mathfrak{o}$ of $L$ representing a nontrivial homology classes on the infinity page of the Turner spectral sequence. Let $\overline{\mathfrak{o}}$ be the reverse orientation of $\mathfrak{o}$. The polynomial gradings of $s_{\mathfrak{o}}$ and $s_{\mathfrak{o}}+s_{\overline{\mathfrak{o}}}$ differ by two, while Lemma \ref{Bar-Natan dimension} implies their homological gradings are the same. In summary, if $P=\Z_2\oplus \Z_2$, where one copy of $\Z_2$ is in bidegree $(0,1)$ and the other is in bidegree $(0,-1)$, then 
\[E_T^{\infty}=\bigoplus_{i=1}^{n-1}P[h_i]\{p_i\}.\]
Because $E^1_T\cong H(L;\Z_2)$, there are nontrivial $\Z_2$ summands in $H^{h_i,p_i\pm1}(L;\Z_2)$.
See the proof of Theorem 3.1 in \cite{turner2006calculating} and Proposition 2.6 in \cite{ls} for details.
\label{pawnmovesturner}


\item The Turner spectral sequence converges to Bar-Natan homology: $E_T^r \implies \BN^*(L)^\prime$, where \[\BN^i(L)^\prime=\bigoplus_{j}(E_T^\infty)^{i,j}.\]
\end{enumerate}

The next ingredient we need is a `vertical' differential on the $\Z_2$-Khovanov complex due to Shumakovitch \cite{shumakovitch2004torsion}. Let $D$ be a diagram of a link $L$. 
A differential $\nu : C(D;\Z_2) \rightarrow C(D;\Z_2)$ is defined as follows. Recall that for a Kauffman state $D(I)$, the algebra $C(D(I))$ has $2^{|D(I)|}$ generators of the form $a_1 \otimes a_2 \otimes \cdots \otimes a_{|D(I)|}$ where $a_i \in \{ 1,X \}$ and $|D(I)|$ is the number of circles in $D(I)$. For each Kauffman state we define a map $\nu_{D(I)} : C(D(I)) \rightarrow C(D(I))$ by sending a generator to the sum of all possible generators obtained by replacing a single $X$ with a $1$. For example, $X \otimes X \otimes X \mapsto 1 \otimes X  \otimes X + X \otimes 1 \otimes X + X \otimes X \otimes 1$. We then extend $\nu_{D(I)}$ linearly to all of $C(D(I))$, and then to a map $\nu : C(D;\Z_2) \rightarrow C(D;\Z_2)$. The properties of this map relevant to our purposes are given below, the proofs of which can be found in Shumakovitch \cite{shumakovitch2004torsion}. 

\begin{enumerate}[(V1)]

\item The map $\nu$ is a differential of bidegree $(0,2)$.

\item The map $\nu$ commutes with the Khovanov differential $d$, and so induces a map (differential) $\nu^* : H(L;\Z_2) \rightarrow H(L;\Z_2)$ on homology.

\item The complex $(H(L;\Z_2), \nu^*)$ is acyclic, that is, it has trivial homology.
\label{viso}

\end{enumerate}

The following lemma of Shumakovitch \cite[Lemma 3.2.A]{shumakovitch2018torsion} is the key to proving Theorem \ref{GeneralTheorem}. 
It relates the first $\mathbb{Z}_2$ Bockstein map with the Turner and vertical differentials. We use the behavior of the Turner spectral sequence and the acyclic homology induced by $\nu$ to determine when the $\mathbb{Z}_2$ Bockstein spectral sequence collapses.

\begin{lem}[Shumakovitch]
\label{TBV}
Let $L$ be a link. The Bockstein, Turner and vertical differentials on the $\Z_2$-Khovanov homology $H(L;\Z_2)$ of $L$ are related by $d_T^* = d_B^1 \circ \nu^* + \nu^* \circ d_B^1$.\label{specseqrelations}
\end{lem}

\subsection{The Lee spectral sequence}

Lee \cite{lee2005endomorphism} defined an endomorphism on Khovanov homology that Rasmussen \cite{rasmussenslice} used to define the $s$ concordance invariant. Let $R$ be either $\Q$ or $\Z_p$ where $p$ is an odd prime. The Lee differential $d_L: C(D;R)\to C(D;R)$ is defined in the same way as the Khovanov differential $d$, but with multiplication and comultiplication given by
\[
	m_L(1 \otimes 1) = 0, \quad m_L(1 \otimes X) = 0, \quad m_L(X \otimes 1) = 0, \quad m_L(X \otimes X) = 1,
\]
\[
	\Delta_L(1) = 0, \quad \Delta_L(X) = 1\otimes 1.
\]
The map $d_L$ satisfies $d^2_L=0$ and $d\circ d_L + d_L \circ d = 0$. The resulting singly graded homology theory is a link invariant, denoted $\operatorname{Lee}^*(L;R)$, and it behaves as follows.
\begin{lem}[Lee]
\label{leehomology}
Let $L$ be an oriented link with $k$ components. Then, $\dim_{R}(\operatorname{Lee}^*(L;R)) = 2^k$. Specifically, if the components of $L$ are $L_1, L_2, \ldots, L_k$, then
\[
	\dim_{R}(\operatorname{Lee}^i(L;R)) = \#\left\{ E \subset \{1,2, \ldots, k \} \left| \ 2 \times \sum_{\ell \in E, m \in E^c} lk(L_\ell,L_m) = i \right. \right\}
\]
where $lk(L_\ell,L_m)$ is the linking number between $L_\ell$ and $L_m$.
\end{lem}

\begin{defn}
Let $D$ be a diagram of a link $L$ and $(C(D;R),d)$ be the $R$-Khovanov complex. The spectral sequence $(E_L^r, d_L^r)_{r \ge 1}$ associated to the double complex $(C(D;R),d,d_L)$ is called the $R$-Lee spectral sequence.
\end{defn}

The $R$-Lee spectral sequence satisfies the following properties.

\begin{enumerate}[(L1)]
\item The first page is $(E_L^1,d_L^1) \cong (H(L;R),d_L^*)$ where $d_L^* : H(L;R) \rightarrow H(L;R)$ is the induced map on homology.

 \item Each map $d_L^r$ is a differential of bidegree $(1,4r)$. \label{diffLee}

\item 
The dimension of the infinity page of the Lee spectral sequence is $2^n$ where $n$ is the number of components of $L$. There is a generator $s_{\mathfrak{o}}$ for each orientation $\mathfrak{o}$ of $L$ representing the nontrivial homology classes on the infinity page of the Lee spectral sequence. Let $\overline{\mathfrak{o}}$ be the reverse orientation of $\mathfrak{o}$. The polynomial gradings of $s_{\mathfrak{o}}+s_{\overline{\mathfrak{o}}}$ and $s_{\mathfrak{o}}-s_{\overline{\mathfrak{o}}}$ differ by two, while Lemma \ref{leehomology} implies their homological gradings are the same. In summary, if $P=R\oplus R$, where one copy of $R$ is in bidegree $(0,1)$ and the other is in bidegree $(0,-1)$, then 
\[E_L^{\infty}=\bigoplus_{i=1}^{n-1}P[h_i]\{p_i\}.\]Because $E^1_L\cong H(L;R)$, there are nontrivial $R$ summands in $H^{h_i,p_i\pm 1}(L;R)$. 
Because $E^1_L\cong H(L;R)$, there is a bigraded submodule $S$ of $H(L;R)$ isomorphic to $E_L^\infty$.
See Section 4.4 in \cite{lee2005endomorphism}, Proposition 3.3 in \cite{rasmussenslice}, and Definition 7.1 in \cite{BW} for details.
\label{pawnmoves}


\item The Lee spectral sequence converges to Lee homology $E_L^r \implies \operatorname{Lee}^*(L;R)$, where $$\operatorname{Lee}^i(L;R)=\bigoplus_{j}(E_L^\infty)^{i,j}.$$
\end{enumerate}

\section{The main result}\label{spectralsequences}\label{TheMainResult}

We now prove several lemmas that will lead to the proofs of Theorems \ref{GeneralTheorem} and \ref{GeneralTheorem2}. Throughout the proofs, we take advantage of the properties of the $\Z_2$-Bockstein spectral sequence, Turner spectral sequence, Lee spectral sequence and vertical differential $\nu^*$, described in Section \ref{ComputationalTools}.

Suppose that $H(L)$ is thin over the interval $[i_1,i_2]$, and let $H^{[i_1,i_2]}(L;R)$ denote the direct sum
\[H^{[i_1,i_2]}(L;R) = \bigoplus_{i=i_1}^{i_2} H^{i,*}(L;R).\]
As before, we drop the $R$ from the notation in the case where $R=\Z$. Our first lemma states that all torsion in homological gradings $(i_1,i_2]$ must be supported on the lower diagonal.
\begin{lem}
\label{lowerdiagonal}
If $H(L)$ is thin over $[i_1,i_2]$ where $H^{[i_1,i_2]}(L)$ is supported in bigradings $(i,j)$ with $2i-j=s\pm 1$ for some $s\in \Z$, then any torsion summand of $H^{[i_1,i_2]}(L)$ with homological gradings $i>i_1$ is supported on the lower diagonal in bigrading $(i,2i-s-1)$.
\end{lem}
\begin{proof}
 If $H^{i,2i-s+1}(L)$ has a nontrivial torsion summand for some $i\in(i_1,i_2]$, then the universal coefficient theorem implies that $H^{i-1,2i-s+1}(L;\Z_p)$ is nontrivial for some $p$, contradicting the fact that $H(L)$ is thin over $[i_1,i_2]$. Therefore all torsion in homological gradings $(i_1,i_2]$ appears in bigradings of the form $(i,2i-s-1)$.
\end{proof}

Our next lemma gives a sufficient condition to ensure that $H^{[i_1,i_2]}(L;\Z)$ has no odd torsion.
\begin{lem}
\label{oddtorsion}
Suppose that $H(L)$ is thin over $[i_1,i_2]$, all Lee differentials are zero in homological grading $i_1-1$, and that 
\[\dim_{\Q}H^{i_1,*}(L;\Q)=\dim_{\Z_p} H^{i_1,*}(L;\Z_p)\]
for each odd prime $p$. Then $H^{[i_1,i_2]}(L;\Z)$ contains no torsion of odd order.
\end{lem}
\begin{proof} 
Let $R$ be $\mathbb{Q}$ or $\Z_p$ for an odd prime $p$. Since $H(L)$ is thin over $[i_1,i_2]$, it follows that there is an $s\in\Z$ such that $H^{[i_1,i_2]}(L)$ is supported in bigradings $(i,j)$ with $2i-j=s\pm 1$.  Lemma \ref{leehomology} implies that the homological gradings of $E^\infty_L$ are determined by the linking numbers of the components of $L$, and thus do not depend on $R$. Because $H(L)$ is thin over $[i_1,i_2]$, Property (L\ref{pawnmoves}) implies that in each homological grading $i\in[i_1,i_2]$ where $E^\infty_L$ is nontrivial, its polynomial gradings are $2i-s-1$ and $2i-s+1$, and hence do not depend on $R$. Therefore, for each $i\in[i_1,i_2]$,
\begin{equation}
    \label{InfintyPage}
\dim_R(E_L^\infty)^{i,2i-s-1}(L;R)=\dim_R(E_L^\infty)^{i,2i-s+1}(L;R).
\end{equation}

We show that if $i\in[i_1,i_2]$, then there cannot be any torsion in $H^{i,*}(L)$ of the form $\Z_{p^r}$ for an odd prime $p$. By way of contradiction suppose that for some $i\in[i_1,i_2]$, the group $H^{i,2i-s-1}(L)$ contains a torsion summand of the form $\Z_{p^r}$ for some $r>0$. The universal coefficient theorem implies that
\begin{align*}
    \dim_{\Q} H^{i,2i-s-1}(L;\Q)< & \; \dim_{\Z_p} H^{i,2i-s-1}(L;\Z_p)~\text{and}\\
      \dim_{\Q} H^{i-1,2i-s-1}(L;\Q)< & \; \dim_{\Z_p} H^{i-1,2i-s-1}(L;\Z_p).
\end{align*}
If $i_1=i$ or $i_1=i-1$, then we have a contradiction, and so we assume that $i_1<i-1$. Equation \eqref{InfintyPage} and the previous inequality imply that the rank of the bidegree $(1,4)$ Lee map in bigrading $(i-2,2i-s-5)$ over $\Z_p$ is greater than its rank over $\Q$. If $i_1=i-2$, then the Lee differential  $(d_L^r)^{i-3,j}$ is trivial on all pages $r>0$ and for all $j$ by assumption, and if $i_1<i-2$, then the Lee differential $(d_L^r)^{i-3,j}$ is trivial on all pages $r>0$ and for all $j\leq 2i-s-9$ because $H(L)$ is thin over $[i_1,i_2]$. Consequently,
\[\dim_{\Q} H^{i-2,2i-s-5}(L;\Q) < \dim_{\Z_p} H^{i-2,2i-s-5}(L;\Z_p).\]
Since $H(L)$ is thin over $[i_1,i_2]$, it follows that $H^{i-1,2i-s-5}(L)=0$, and in particular, there is no $\Z_{p^r}$ torsion summand in bigrading $(i-1,2i-s-5)$. Thus the universal coefficient theorem implies there is a $\Z_{p^{r'}}$ torsion summand in bigrading $(i-2,2i-s-5)$ for some $r'>0$. In summary, a $\Z_{p^r}$ torsion summand in homological grading $i$ induces a dimension inequality in homological gradings $i, i-1$, and $i-2$, and it induces a $\Z_{p^{r'}}$ summand in homological grading $i-2$. Since $i\in[i_1,i_2]$ is arbitrary, repeating this argument for each new torsion summand implies that 
\[\dim_{\Q} H^{i_1,*}(L;\Q) < \dim_{\Z_p} H^{i_1,*}(L;\Z_p),\]
which is a contradiction.
\end{proof}

Our next lemma states that in a thin region, the induced Turner map $d_T^*$ has the same rank on the lower and upper diagonals. 
\begin{lem}
\label{Turnerrank}
Suppose that $H(L)$ is thin over $[i_1,i_2]$ where $H^{[i_1,i_2]}(L)$ is supported in bigradings $(i,j)$ with $2i-j=s\pm 1$ for some $s\in\mathbb{Z}$. If $i\in[i_1,i_2)$, then
\[\rk (d_T^*)^{i,2i-s-1} = \rk (d_T^*)^{i,2i-s+1}.\]
\end{lem}
\begin{proof} Lemma \ref{TBV} states that $d_T^*=d_B^1\circ \nu^*+\nu^*\circ d_B^1$. Because $\nu^*\circ \nu^*=0$, it follows that
\[\nu^*\circ d_T^* = \nu^*\circ d_B^1 \circ \nu^* = d_T^*\circ \nu^*.\]
Property (V\ref{viso}) and the fact that $H(L)$ is thin over $[i_1,i_2]$ imply that $\nu^*$ is an isomorphism in homological gradings $[i_1,i_2]$, which then implies the desired result.
\end{proof}

Suppose that $H(L)$ is thin over $[i_1,i_2]$ where $H^{[i_1.i_2]}(L)$ is supported in bigradings $(i,j)$ with $2i-j = s\pm 1$ for some $s\in\mathbb{Z}$. Also, suppose that all Lee differentials in homological grading $i_1-1$ are trivial and that $\dim_{\Q}H^{i_1,*}(L;\Q) = \dim_{\Z_p} H^{i_1,*}(L;\Z_p)$. Lemmas \ref{lowerdiagonal} and \ref{oddtorsion} imply that the upper diagonal of $H^{[i_1,i_2]}(L)$ is torsion-free and that all torsion on the lower diagonal is of the form $\Z_{2^r}$ for various different values of $r$. Therefore, $H^{i,2i-s+1}(L) \cong \Z^{k_i^+}$ and 
\[ H^{i,2i-s-1} (L)\cong \Z^{k_i^-} \oplus \Z_{2^{r_1}}\oplus \cdots \oplus \Z_{2^{r_{\ell_i}}}\]
for some $k_i^+, k_i^-, \ell_i,r_1,\dots, r_{\ell_i}\in\mathbb{Z}$ and for each $i\in[i_1,i_2]$. Figure \ref{thinlabels} represents the different summands in the thin region. Lemmas \ref{Bar-Natan dimension} and \ref{leehomology}, together with (T\ref{pawnmovesturner}) and (L\ref{pawnmoves}) imply that $\dim_{\Q} (E_L^\infty)^{i,2i-s\pm1}=\dim_{\Z_2}(E_T^\infty)^{i,2i-s\pm 1}$ when $i\in[i_1,i_2]$. For each $i\in[i_1,i_2]$, let $n_i = \dim_{\Q} (E_L^\infty)^{i,2i-s\pm1} =\dim_{\Z_2}(E_T^\infty)^{i,2i-s\pm 1}$. The notation established in this paragraph will be used in the statements and proofs of Lemmas \ref{Turner2} and \ref{TurnerLeeTorsion}.

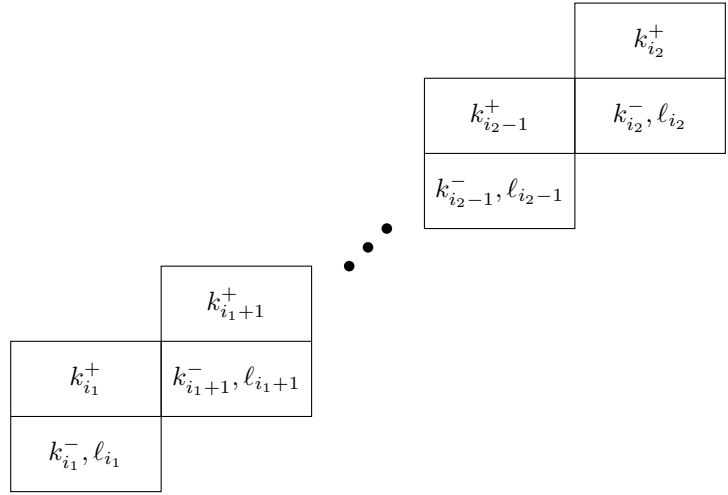
\begin{figure}[ht]
\[\begin{tikzpicture}[scale = 1]

   \draw (0,0) rectangle (2,2);
    \draw (0,1) -- (2,1);
    \draw (2,1) rectangle (4,3);
    \draw (2,2) -- (4,2);
    \draw (1, 0.5) node{$k_{i_1}^-,\ell_{i_1}$};
    \draw (1,1.5) node{$k_{i_1}^+$};
    \draw (3,1.5) node{$k_{i_1+1}^-,\ell_{i_1+1}$};
    \draw (3,2.5) node{$k_{i_1+1}^+$};

    \fill (4.75,3.25) circle (2pt);
    \fill (4.5,3.0) circle (2pt);
    \fill (5,3.5) circle (2pt);

    \begin{scope}[xshift = 5.5cm, yshift = 3.5cm]
   \draw (0,0) rectangle (2,2);
    \draw (0,1) -- (2,1);
    \draw (2,1) rectangle (4,3);
    \draw (2,2) -- (4,2);
    \draw (1, 0.5) node{$k_{i_2-1}^-,\ell_{i_2-1}$};
    \draw (1,1.5) node{$k_{i_2-1}^+$};
    \draw (3,1.5) node{$k_{i_2}^-,\ell_{i_2}$};
    \draw (3,2.5) node{$k_{i_2}^+$};
  
    \end{scope}

\end{tikzpicture}\]
\caption{A depiction of the summands of $H^{[i_1,i_2]}(L)$. The entry $k_{i}^\pm$ represents the summand $\Z^{k_{i}^\pm}$ in bigrading $(i,2i-s\pm1)$. The entry $\ell_i$ represents $\ell_i$ summands of the form $\Z_{2^r}$ for various values of $r$ in bigrading $(i,2i-s-1)$. Since $H^{i_1,*}(L)$ has no torsion, $\ell_{i_1}=0$.}
\label{thinlabels}
\end{figure} 

Since the Turner and Lee spectral sequences use coefficients that are fields, it follows that the infinity page of each spectral sequence is determined by the first page and the ranks of certain differentials in the respective spectral sequences. Specifically, for any $i$ and $j$ in $\Z$, 
\begin{align}
\label{TurnerInfinity}
\dim_{\Z_2} (E_T^\infty)^{i,j} = & \; \dim_{\Z_2} (E_T^1)^{i,j} - \sum_{r=1}^\infty \left(\rk (d_T^r)^{i,j} + \rk (d_T^r)^{i-1,j-2r}\right)~\text{and}\\
\label{LeeInfinity}
\dim_\Q (E_L^\infty)^{i,j} = & \; \dim_{\Q} (E_L^1)^{i,j} - \sum_{r=1}^\infty \left(\rk (d_L^r)^{i,j} + \rk (d_L^r)^{i-1,j-4r}\right).
\end{align}
The proofs of Lemmas \ref{Turner2} and \ref{TurnerLeeTorsion} below make use of these formulas along with the fact that most of the maps in the infinite sums above are trivial for grading reasons.

Our next lemma states that the Turner spectral sequence collapses at the second page in a thin region if all incoming Turner differentials are trivial. If the entire homology is thin, then this lemma gives a proof of property (T\ref{Turnercollapse}).

\begin{lem}
\label{Turner2}
Suppose that $H(L)$ is thin over $[i_1,i_2]$ for integers $i_1$ and $i_2$ and that $(d_T^r)^{i_1-1,j}=0$ for all $r\geq 1$ and for all $j\in\Z$. If $i\in[i_1,i_2)$, then $(d_T^r)^{i,j}=0$ for all $r\geq 2$ and for all $j\in\Z$.
\end{lem}
\begin{proof}
Since $H(L)$ is thin over $[i_1,i_2]$, there is an $s\in\Z$ such that $H^{[i_1,i_2]}(L)$ is supported in bigradings $(i,j)$ with $2i-j=s\pm 1$. 
Also, since $H(L)$ is thin over $[i_1,i_2]$ and $d_T^r$ has bidegree $(1,2r)$, it follows that if $r\geq 3$, then the map $(d_T^r)^{i,j}=0$ for all $i\in[i_1,i_2)$ and for all $j\in\Z$ because either its domain or codomain is trivial. 
Suppose that the Turner differential on the second page is nonzero somewhere in the thin region. 
Let $i$ be the minimum homological grading in $[i_1,i_2)$ such that $(d_T^2)^{i,j}$ is nonzero for some $j$. 
For grading reasons, it must be the case that the domain of $(d_T^2)^{i,j}$ is on the bottom diagonal while the codomain of $(d_T^2)^{i,j}$ is on the top diagonal. 
Thus $j=2i-s-1$, and Equation \ref{TurnerInfinity} implies
\begin{align*}
\dim (E_T^\infty)^{i,2i-s-1} = & \; \dim_{\Z_2} (E_T^1)^{i,2i-s-1} - \rk (d_T^1)^{i-1,2i-s-3} - \rk (d_T^1)^{i,2i-s-1} - \rk (d_T^2)^{i,2i-s-1}~\text{and}\\
\dim (E_T^\infty)^{i,2i-s+1} = & \; \dim_{\Z_2} (E_T^1)^{i,2i-s+1} - \rk(d_T^1)^{i-1,2i-s-1} - \rk (d_T^1)^{i,2i-s+1},
\end{align*}
where in the second equation, there is no $\rk (d_T^2)^{i-1,2i-s-3}$ term because $i$ is the minimum homological grading in $[i_1,i_2)$ where $(d_T^2)^{i,j}$ is nontrivial. Property (T\ref{pawnmovesturner}) implies that $\dim(E_T^\infty)^{i,2i-s-1} = \dim(E_T^\infty)^{i,2i-s+1}$, and since $(E_T^1)^{i,j}\cong H^{i,j}(L;\Z_2)$, property (V\ref{viso}) implies that $\dim_{\Z_2} (E_T^1)^{i,2i-s-1} = \dim_{\Z_2} (E_T^1)^{i,2i-s+1}.$

If $i=i_1$, then $\rk (d_T^1)^{i-1,2i-s-3}= \rk(d_T^1)^{i-1,2i-s-1} =0$ by assumption, and $\rk (d_T^1)^{i_1,2i_1-s-1} = \rk (d_T^1)^{i_1,2i_1-s+1}$ by Lemma \ref{Turnerrank}. 
If $i>i_1$, then Lemma \ref{Turnerrank} implies that $\rk (d_T^1)^{i-1,2i-s-3} = \rk(d_T^1)^{i-1,2i-s-1}$ and $\rk (d_T^1)^{i,2i-s-1} = \rk (d_T^1)^{i,2i-s+1}$. Since $\rk (d_T^2)^{i,2i-s-1} > 0$ by assumption, we have a contradiction in either case. 
Therefore $(d_T^2)^{i,j}$ is zero for all $i\in[i_1,i_2)$ and for all $j\in\Z$.
\end{proof}

Our next lemma shows a relationship between the ranks of the Turner and Lee differentials in a thin region and also between the number of torsion summands of the form $\Z_{2^r}$ for various values of $r$.
\begin{lem}
\label{TurnerLeeTorsion}
Suppose that a link $L$ satisfies:
\begin{enumerate}
    \item $H(L)$ is thin over $[i_1,i_2]$ for integers $i_1$ and $i_2$,
    \item $H^{i_1,*}(L)$ is torsion-free, and
    \item all Lee and Turner differentials are zero in homological grading $i_1-1$ on every page of the respective spectral sequences.
\end{enumerate}
Then for each $i\in[i_1,i_2)$,
\[\rk (d_T^*)^{i,2i-s-1} = \rk (d_L^*)^{i,2i-s-1} = \ell_{i+1},\]
where $\ell_{i+1} = \dim_{\Z_2} H^{i+1,2i-s+1}(L)\otimes\Z_2 - \rk H^{i+1,2i-s+1}(L)$  is the number of summands of the form $\Z_{2^r}$ for various values of $r$ in $H^{i+1,2i-s+1}(L)$. 
\end{lem}
\begin{proof}
Since $H(L)$ is thin over $[i_1,i_2]$, there is an $s\in\Z$ such that $H^{[i_1,i_2]}(L)$ is supported in bigradings $(i,j)$ with $2i-j=s\pm 1$. We proceed by induction on the homological grading $i$.

For the base case, we prove the desired statement when $i=i_1$. The left side of Figure \ref{mapfigure} depicts the maps involved in the base case. All incoming Lee maps are trivial, and thus Equation \ref{LeeInfinity} implies $\rk H^{i_1,2i_1-s+1}(L) = k_{i_1}^+=n_{i_1} = \dim_{\Q} (E_L^\infty)^{i_1,2i_1-s+1}$. The dimension of $H^{i_1,2i_1-s+1}(L;\Z_2)$ is $n_{i_1}+\ell_{i_1+1}$.  Equation \ref{TurnerInfinity} implies
\[n_ {i_1} = \dim_{\Z_2} (E_T^\infty)^{i_1,2i_1-s+1} = n_{i_1} + \ell_{i_1+1} - \rk (d_T^*)^{i_1,2i_1-s+1},\]
and thus $\rk(d_T^*)^{i_1,2i_1-s+1} = \ell_{i_1+1}$. Lemma \ref{Turnerrank} then implies $\rk(d_T^*)^{i_1,2i_1-s-1}=\rk(d_T^*)^{i_1,2i_1-s+1} = \ell_{i_1+1}$. 

Since $\dim_\Q (E_L^\infty)^{i_1,2i_1-s-1} = n_{i_1}$ and all Lee maps in homological grading $i_1-1$ are trivial, Equation \ref{LeeInfinity} implies that $\rk (d_L^*)^{i_1,2i_1-s-1} = \rk H^{i_1,2i_1-s-1}(L) - \dim_{\Q} (E_L^\infty)^{i_1,2i_1-s-1} = k_{i_1}^- - n_{i_1}$. Because $H^{i_1,*}(L)$ is torsion free, $\ell_{i_1}=0$ , and thus $\dim_{\Z_2}H^{i_1,2i_1-s-1}(L;\Z_2) = \dim_{\Z_2} (E_T^1)^{i_1,2i_1-s-1} = k_{i_1}^-$. Lemma \ref{Turner2} implies that $(d_T^2)^{i_1,2i_1-s-1} = 0$, and since all Turner maps in homological grading $i_1-1$ are trivial, Equation \ref{TurnerInfinity} implies that
$\rk (d_T^*)^{i_1,2i_1-s-1} = \dim_{\Z_2}H^{i_1,2i_1-s-1}(L;\Z_2) - \dim_{\Z_2} (E_T^\infty)^{i_1,2i_1-s-1} = k_{i_1}^- - n_{i_1}.$ Therefore $\rk (d_T^*)^{i_1,2i_1-s-1} = \rk (d_L^*)^{i_1,2i_1-s-1}$, which completes the base case.

For the inductive step, let $i\in (i_1,i_2)$ and assume that $\rk(d_T^*)^{i-1,2i-s-3}=\rk(d_L^*)^{i-1,2i-s-3}=\ell_i$. The right side of Figure \ref{mapfigure} depicts the maps involved in the inductive step. 
Because the Lee spectral sequence collapses at the second page in bigrading $(i,2i-s+1)$, Equation \ref{LeeInfinity} implies that $\dim_{\Q} (E_L^\infty)^{i,2i-s+1} = n_i = k_i^+ - \rk (d_L^*)^{i-1,2i-s-3}$. 
Similarly, since Lemma \ref{Turner2} implies that the Turner spectral sequence collapses at the second page in bigrading $(i,2i-s+1)$, Equation \ref{TurnerInfinity} implies that $\dim_{\Z_2} (E_T^\infty)^{i,2i-s+1}=n_i = k_i^+ + \ell_{i+1} - \rk(d_T^*)^{i-1,2i-s-1} - \rk(d_T^*)^{i,2i-s+1}.$ 
Therefore, in our notation, the equation $\dim_{\Z_2}(E_T^\infty)^{i,2i-s\pm 1}=n_i=\dim_{\Q}(E_L^\infty)^{i,2i-s\pm 1}$ for all bigradings with $i_1 < i < i_2$ can be written in the following way:
\begin{equation}
    \label{4.5eq}
    k_i^+ - \rk (d_L^*)^{i-1,2i-s-3} = k_i^+ + \ell_{i+1} - \rk(d_T^*)^{i-1,2i-s-1} - \rk(d_T^*)^{i,2i-s+1}.
\end{equation}
The inductive hypothesis implies that $\rk (d_L^*)^{i-1,2i-s-3} = \rk(d_T^*)^{i-1,2i-s-3}$, and Lemma \ref{Turnerrank} implies that $\rk(d_T^*)^{i-1,2i-s-3}=\rk(d_T^*)^{i-1,2i-s-1}$. Therefore $\rk (d_L^*)^{i-1,2i-s-3} = \rk(d_T^*)^{i-1,2i-s-1}$, and thus Equation \ref{4.5eq} implies that $\rk(d_T^*)^{i,2i-s+1} = \ell_{i+1}$. Lemma \ref{Turnerrank} then implies that  $\rk(d_T^*)^{i,2i-s-1}=\rk(d_T^*)^{i,2i-s+1} = \ell_{i+1}$. 

Because the Lee spectral sequence collapses at the second page in bigrading $(i,2i-s-1)$, Equation \ref{LeeInfinity} implies that $\dim_{\Q} (E_L^\infty)^{i,2i-s-1} =n_i=k_i^- -\rk(d_L^*)^{i,2i-s-1}$. Similarly, because Lemma \ref{Turner2} implies that the Turner spectral sequence collapses at the second page in bigrading $(i,2i-s-1)$, Equation \ref{TurnerInfinity} implies that $\dim_{\Z_2} (E_T^\infty)^{i,2i-s-1} = n_i=k_i^-+\ell_i - \rk(d_T^*)^{i-1,2i-s-3} - \rk(d_T^*)^{i,2i-s-1}$. The inductive hypothesis states that $\rk(d_T^*)^{i-1,2i-s-3} = \ell_i$, and thus $n_i=k_i^- - \rk(d_T^*)^{i,2i-s-1}$. Therefore $ \rk(d_T^*)^{i,2i-s-1} = \rk(d_L^*)^{i,2i-s-1}$, completing the proof.
\end{proof}
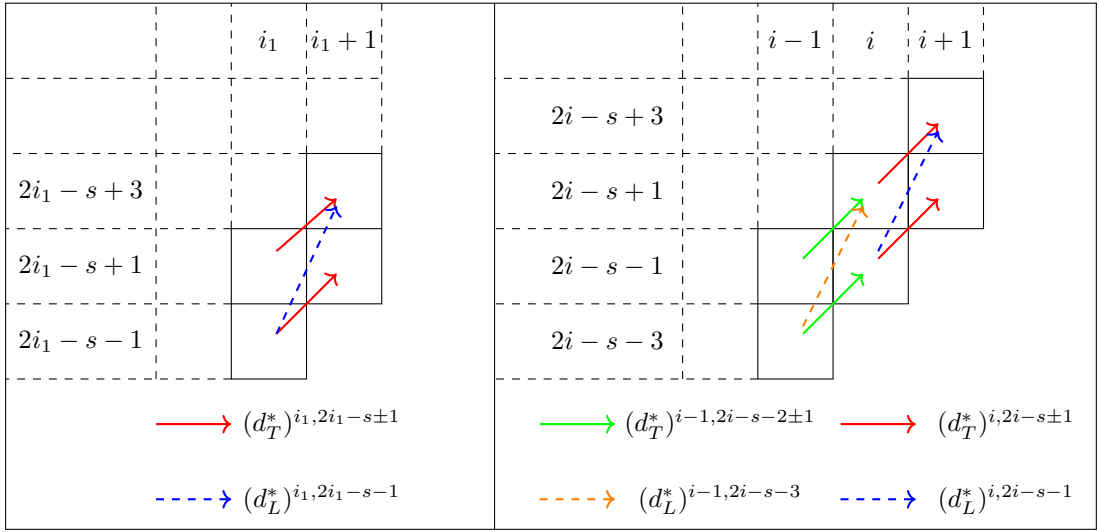
\begin{figure}
\[\begin{tikzpicture}[scale=1]
\draw (0,0) rectangle (1,2);
\draw (1,1) rectangle (2,3);
\draw (0,1) -- (1,1);
\draw (1,2) -- (2,2);
\draw[dashed] (0,0) -- (-3,0);
\draw[dashed] (0,1) -- (-3,1);
\draw[dashed] (0,2) -- (-3,2);
\draw[dashed] (1,3) -- (-3,3);
\draw[dashed] (-1,0) -- (-1,5);
\draw[dashed] (0,2) -- (0,5);
\draw[dashed] (1,3) -- (1,5);
\draw[dashed] (2,3) -- (2,5);
\draw[dashed] (-3,4) -- (2,4);

\draw (-2,.5) node{$2i_1-s-1$};
\draw (-2,1.5) node{$2i_1-s+1$};
\draw (-2,2.5) node{$2i_1-s+3$};
\draw (.5,4.5) node{$i_1$};
\draw (1.5,4.5) node{$i_1+1$};

\draw[red, thick, ->] (.6,.6) -- (1.4,1.4);
\draw[red,thick, ->] (.6,1.7) -- (1.4,2.4);
\draw[blue, dashed, thick, ->] (.6,.6) -- (1.4,2.3);

\begin{scope}[xshift = -1cm, yshift = .4cm]
\draw[thick, red,->] (0,-1) -- (1,-1);
\draw (2.2,-1) node{$(d_T^*)^{i_1,2i_1-s\pm 1}$};
\draw[thick,dashed, blue, ->] (0,-2) -- (1,-2);
\draw (2.2,-2) node{$(d_L^*)^{i_1,2i_1-s-1}$};
\end{scope}

\begin{scope}[xshift = 7cm]
\draw (0,0) rectangle (1,2);
\draw (1,1) rectangle (2,3);
\draw (2,2) rectangle (3,4);
\draw (0,1) -- (1,1);
\draw (1,2) -- (2,2);
\draw (2,3) -- (3,3);
\draw[dashed] (0,0) -- (-3.5,0);
\draw[dashed] (0,1) -- (-3.5,1);
\draw[dashed] (0,2) -- (-3.5,2);
\draw[dashed] (1,3) -- (-3.5,3);
\draw[dashed] (2,4) -- (-3.5,4);
\draw[dashed] (-1,0) -- (-1,5);
\draw[dashed] (0,2) -- (0,5);
\draw[dashed] (1,3) -- (1,5);
\draw[dashed] (2,4) -- (2,5);
\draw[dashed] (3,4) -- (3,5);

\draw (-2,.5) node{$2i-s-3$};
\draw (-2,1.5) node{$2i-s-1$};
\draw (-2,2.5) node{$2i-s+1$};
\draw (-2,3.5) node{$2i-s+3$};
\draw (.5,4.5) node{$i-1$};
\draw (1.5,4.5) node{$i$};
\draw (2.5,4.5) node{$i+1$};

\draw[red, thick, ->] (1.6,1.6) -- (2.4,2.4);
\draw[red, thick, ->] (1.6,2.6) -- (2.4,3.4);
\draw[blue, dashed, thick, ->] (1.6,1.7) -- (2.4,3.3);

\draw[green, thick, ->] (.6,.6) -- (1.4,1.4);
\draw[green, thick, ->] (.6,1.6) -- (1.4,2.4);
\draw[orange,dashed, thick,->] (.6,.7) -- (1.4,2.3);
\begin{scope}[xshift=-.4cm,yshift=.4cm]

\draw[green, thick, ->] (-2.5,-1) -- (-1.5,-1);
\draw[orange, dashed, thick, ->] (-2.5,-2) -- (-1.5,-2);

\draw (-.1,-1) node{$(d_T^*)^{i-1,2i-s-2\pm1}$};
\draw (-.1,-2) node{$(d_L^*)^{i-1,2i-s-3}$};

\draw[red, thick, ->] (1.5,-1) -- (2.5,-1);
\draw[blue, dashed, thick, ->] (1.5,-2) -- (2.5,-2);

\draw (3.7,-1) node{$(d_T^*)^{i,2i-s\pm 1}$};
\draw (3.7,-2) node{$(d_L^*)^{i,2i-s-1}$};
\end{scope}
\end{scope}

\draw (-3,-2) rectangle (11.5,5);
\draw (3.5,-2) -- (3.5,5);
\end{tikzpicture}\]
    \caption{The left side depicts the Turner and Lee maps involved in the base case of the proof of Lemma \ref{TurnerLeeTorsion}. The right side depicts the Turner and Lee maps involved in the inductive step of the proof of Lemma \ref{TurnerLeeTorsion}.}
    \label{mapfigure}
\end{figure}

We can now combine the previous lemmas to give sufficient conditions for all torsion in a thin region to be of order two.
\begin{thm}
\label{GeneralTheorem}
Suppose that a link $L$ satisfies:
\begin{enumerate}
    \item $H(L)$ is thin over $[i_1,i_2]$ for integers $i_1$ and $i_2$,
    \item $\dim_{\Q} H^{i_1,*}(L;\Q)=\dim_{\Z_p} H^{i_1,*}(L;\Z_p)$ for each odd prime $p$,
    \item $H^{i_1,*}(L)$ is torsion-free, and
    \item all Lee and Turner differentials are zero in homological grading $i_1-1$ on every page of the respective spectral sequences.
\end{enumerate}
If $i\in[i_1,i_2]$, then all torsion in $H^{i,*}(L)$ is $\Z_2$ torsion, that is, $H^{[i_1,i_2]}(L)\cong\Z^k\oplus \Z_2^\ell$ for some $k,\ell\geq 0$.
\end{thm}

\begin{proof}
Since $H(L)$ is thin over $[i_1,i_2]$, there is an integer $s$ such that $H^{[i_1,i_2]}(L)$ is supported in bigradings $(i,j)$ satisfying $2i-j=s\pm 1$. Lemma \ref{lowerdiagonal} implies that all torsion in $H^{[i_1,i_2]}(L)$ occurs on the lower diagonal, i.e. in bigradings $(i,2i-s-1)$ for $i\in(i_1,i_2]$. Lemma \ref{oddtorsion} implies that $H^{[i_1,i_2]}(L)$ does not contain any torsion summands of $\Z_{p^r}$ for any odd prime $p$. Therefore $H^{[i_1,i_2]}(L)$ consists of $\Z$ and $\Z_{2^r}$ summands for various values of $r$.

Lemmas \ref{TBV} and \ref{TurnerLeeTorsion} imply that $\rk (d_B^1)^{i,2i-s+1} = \rk (d_T^*)^{i,2i-s-1} = \ell_{i+1}$ for each $i\in[i_1,i_2)$. Property (B\ref{bockprop}) of the Bockstein spectral sequence implies that there is no torsion in $H^{[i_1,i_2]}(L)$ of order $2^r$ for $r>1$. Therefore, the only torsion in $H^{i,*}(L)$ is of the form $\Z_2$ for $i\in [i_1,i_2]$.
\end{proof}

The main theorem of the paper follows from Theorem \ref{GeneralTheorem}.
\begin{thm}
\label{GeneralTheorem2}
Suppose that a link $L$ satisfies:
\begin{enumerate}
    \item $H(L)$ is thin over $[i_1,i_2]$ for integers $i_1$ and $i_2$ where $H^{[i_1,i_2]}(L)$ is supported in bigradings $(i,j)$ with $2i-j=s\pm 1$ for some $s\in\mathbb{Z}$,
    \item $\dim_{\Q}H^{i_1,*}(L;\Q)=\dim_{\Z_p} H^{i_1,*}(L;\Z_p)$ for each odd prime $p$,
    \item $H^{i_1,*}(L)$ is torsion-free, and
    \item $H^{i_1-1,j}(L)$ is trivial when $j\leq 2i_1-s -3.$
\end{enumerate}
Then all torsion in $H^{i,*}(L)$ is $\Z_2$ torsion for $i\in[i_1,i_2]$, that is, $H^{[i_1,i_2]}(L)=\Z^k\oplus \Z_2^\ell$ for some $k,\ell\geq 0$. 
\end{thm}

\begin{proof}
Property (L\ref{diffLee}) states that the Lee differential on the $E^r_L$ page has bidegree $(1,4r)$. Since $H^{i_1-1,j}(L)$ is trivial when $j\leq 2i_1-s -3$, all Lee differentials in homological grading $i_1-1$ are zero. Property (T\ref{diffTurner}) states that the Turner differential on the $E^r_T$ page has bidegree $(1,2r)$. Thus the only potential nonzero differential is $(d_T^*)^{i_1-1,2i_1-s-1}$ from $H^{i_1-1,2i_1-s-1}(L;\Z_2)$ to $H^{i_1,2i_1-s+1}(L;\Z_2)$. By Lemma \ref{TBV}, we have $(d_T^*)^{i_1-1,2i_1-s-1} = (d_B^1)^{i_1-1,2i_1-s+1} \circ (\nu^*)^{i_1-1,2i_1-s-1} + (\nu^*)^{i_1,2i_1-s-1} \circ (d_B^1)^{i_1-1,2i_1-s-1}$. If $(d_T^*)^{i_1-1,2i_1-s-1}$ is nonzero, then at least one of $(d_B^1)^{i_1-1,2i_1-s+1}$ or $(d_B^1)^{i_1-1,2i_1-s-1}$ is also nonzero, contradicting the fact that $H^{i_1,*}(L)$ has no torsion. Thus $(d_T^*)^{i_1-1,2i_1-s-1} =0$. Therefore, all Turner differentials in homological grading $i_1-1$ are zero. The result follows from Theorem \ref{GeneralTheorem}.
\end{proof}

\section{An application to 3-braids}\label{application}

There are a number of results about the Khovanov homology of closed $3$-braids, but a full computation of the Khovanov homology of closed $3$-braids remains open. Turner \cite{turner2008spectral} computed the Khovanov homology of the $(3,q)$ torus links $T(3,q)$ over coefficients in $\Q$ or $\Z_p$ for an odd prime $p$ (see also Sto\v{s}i\'{c} \cite{stosic}). Benheddi \cite{benheddi2017khovanov} computed the reduced Khovanov homology of $T(3,q)$ with coefficients in $\Z_2$. Let $\widetilde{H}(L;\Z_2)$ be the reduced Khovanov homology of $L$ with $\Z_2$ coefficients. The Khovanov homology of $T(3,q)$ with coefficients in $\Z_2$, shown here in Figure \ref{QversusZ2}, can be obtained from Benheddi's computations via the isomorphism 
\begin{equation}
    \label{reducediso}
H^{i,j}(L;\Z_2) \cong \widetilde{H}^{i,j-1}(L;\Z_2) \oplus \widetilde{H}^{i,j+1}(L;\Z_2)
\end{equation}
from Corollary 3.2.C in \cite{shumakovitch2004torsion}. We recover our $j$ grading from Benheddi's $\delta$ grading by letting $j=\delta+2i$. Both Turner and Benheddi's computations play a crucial role in our proofs.

The literature on the Khovanov homology of non-torus closed $3$-braids is considerably more sparse. Baldwin \cite{Baldwin} proved that a closed $3$-braid is quasi-alternating if and only if its Khovanov homology is homologically thin. Abe and Kishimoto \cite{abekishimoto} used the Rasmussen $s$-invariant to compute the alternation number and dealternating number of many closed $3$-braids. Lowrance \cite{Lowrance3braid} computed the homological width of the Khovanov homology of all closed $3$-braids.

Over the next few sections, we prove Theorem \ref{mainthm2}, showing that all torsion in the Khovanov homology of a closed braid in $\Omega_0, \Omega_1, \Omega_2$, or $\Omega_3$ is $\Z_2$ torsion. In this section, we use the same notation for a braid and its closure when the context is clear. First, we argue that it suffices to prove Theorem \ref{mainthm2} when the exponent $n$ in $\Delta^n$ in the braid word is non-negative. 
Using Turner's \cite{turner2008spectral} and Benheddi's \cite{benheddi2017khovanov} computations together with the long exact sequences \ref{negativeLES} and \ref{positiveLES}, we obtain the Khovanov homology for all links in $\Omega_0, \Omega_1,\Omega_2$ and $\Omega_3$, over $\Q$ and $\Z_p$ where $p$ is any prime. Finally, we use these computations together with Theorem \ref{GeneralTheorem2} to obtain the integral Khovanov homology.

\subsection{Reducing to the case $n\geq 0$}

Murasugi's classification of $B_3$ expresses any 3-braid as a word $\Delta^{2n} \beta$ for some $n\in\Z,\beta\in B_3$, up to conjugation. The following observations imply that, for the purposes of determining which possible types of torsion which may appear, we can assume $n\geq 0$.
\begin{enumerate}
\item The mirror image  $m(D)$ of a link diagram $D$ is the diagram obtained by changing all crossings. On the level of braid words, $m : B_3 \rightarrow B_3$ is a group homomorphism satisfying $m(\sigma_i)=\sigma_i^{-1}$ and $m(\Delta) = \Delta^{-1}$. Recall that the torsion in Khovanov homology of a link diagram and the torsion of its mirror image differ only by a homological shift \cite[Corollary 11]{khovanov1999categorification}. So the Khovanov homology of $L$ has $\Z_{p^r}$ torsion if and only if the Khovanov homology of its mirror $m(L)$ has $\Z_{p^r}$ torsion. 

\item Consider the group homomorphism $\phi:B_3\to B_3$ defined on generators by $\phi(\sigma_1)=\sigma_2$ and $\phi(\sigma_2)=\sigma_1$. If the braid word $\omega$ is a projection of a link $L$ embedded in $\{(x,y,z)\in\R^3 \ | \ 0<z<1\}$ to the plane $z=0$, then the projection of $L$ to the plane $z=1$ is $\phi(\omega)$. Thus the map $\phi$ preserves the isotopy type of the braid word. Therefore the Khovanov homology of the closure of $\omega$ has $\Z_{p^r}$ torsion if and only if the Khovanov homology of the closure of $\phi(\omega)$ has $\Z_{p^r}$ torsion. Note that the homomorphism $\phi$ satisfies $\phi(\Delta)=\Delta$. See Figure \ref{dandphiofd} for an example of the action of $\phi$ on a braid diagram.
\end{enumerate}

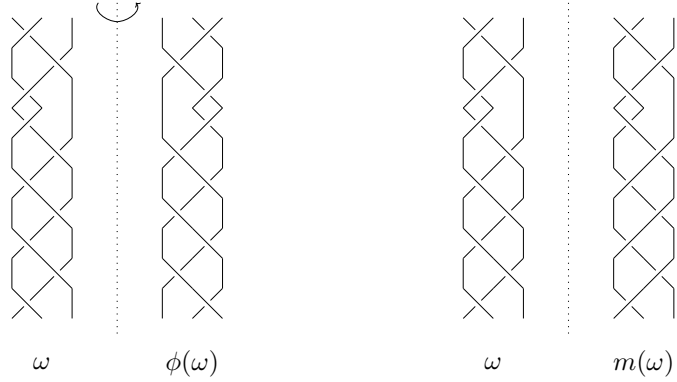
\begin{figure}[ht]
\centering
\begin{tikzpicture}[scale=.4]
\node at (0,0){\usebox\mirrorbraid};
\end{tikzpicture}
\caption{Left: The braid word $\omega = \Delta^2\sigma_1^2\sigma_2\sigma_1 \in \Omega_6$ and the corresponding diagram $\phi(\omega)= \Delta^2\sigma_2^2\sigma_1\sigma_2$. Think of $\phi(D)$ as $D$ rotated about the dotted line. Right: The braid word $\omega$ and its mirror image $m(\omega)$.}
\label{dandphiofd}
\end{figure}

The following equalities together with the above two arguments show that in all cases it suffices to determine torsion for $n\geq0$:
\begin{align}
m(\Delta^{-2n})&=\Delta^{2n} \label{simplifyone}\\
m(\Delta^{-2n}\sigma_1\sigma_2)&=\Delta^{2n-2}(\sigma_1\sigma_2)^2\label{simplifytwo}\\
m(\Delta^{-2n}(\sigma_1\sigma_2)^2)&=\Delta^{2n-2}(\sigma_1\sigma_2)\label{simplifythree}\\
m(\Delta^{-2n-1})&=\Delta^{2n+1}\\
m\phi(\Delta^{-2n}\sigma_1^{-p})&=\Delta^{2n}\sigma_2^p\\
m\phi(\Delta^{-2n}\sigma_2^{q})&=\Delta^{2n}\sigma_1^{-q}\\
m\phi(\Delta^{-2n}\sigma_1^{-p_1}\sigma_2^{q_1}\dots \sigma_1^{-p_r}\sigma_2^{q_r})&=\Delta^{2n}\sigma_2^{p_1}\sigma_1^{-q_1}\dots \sigma_2^{p_r}\sigma_1^{-q_r}.\label{omegasixeqn}
\end{align}
 
 For the case of $\Omega_6$, although we will not address it in this paper, it may be necessary to enlarge the class $\Omega_6$ to a class $\Omega_6^\prime$ which allows powers $p_i,q_j$ to be equal to zero, so that the right hand side of (\ref{omegasixeqn}) stays inside $\Omega_6^\prime$.

\subsection{Odd torsion in $\Omega_0,\Omega_1,\Omega_2,\Omega_3$}

We begin with a theorem, shown by Turner in \cite{turner2008spectral}, that will be useful in conjunction with Murasugi's classification of $3$-braids and the long exact sequence of Section \ref{ComputationalTools}.
\begin{thm}[Turner]\label{toruslinksoddtorsion}
For each $q \in \mathbb{Z}$, the Khovanov homology $H(T(3,q))$ of the torus link $T(3,q)$ contains no $\Z_{p^r}$ torsion for $p \neq 2$. That is, there is no $\Z_{p^r}$ torsion for $p \neq 2$ in the Khovanov homology of links of types $\Omega_0,\Omega_1$ and $\Omega_2$.
\end{thm}


We now compute the Khovanov homology of closed $3$-braids in $\Omega_3$ over $\Q$ or $\Z_p$ for any prime $p$. A corollary of this computation is that all torsion in the Khovanov homology of such links is of the form $\Z_{2^r}$.

\begin{thm}\label{omega3comp} 
For $\F=\Q$ or $\Z_p$ for any odd prime $p$, and any $n\geq 0$,  
\[H(\Delta^{2n+1};\F)\cong H(T(3,3n+1);\F)\{-1\}\oplus H(U;\F)[-4n-2]\{-12n-5\}.\]
\end{thm}

\begin{proof}
First observe that $\Delta^{2n+1}=(\sigma_1\sigma_2)^{3n+1}\sigma_1$ where $(\sigma_1\sigma_2)^{3n+1}$ is a braid word for $T(3,3n+1)$.  We consider smoothing the top $\sigma_1$.
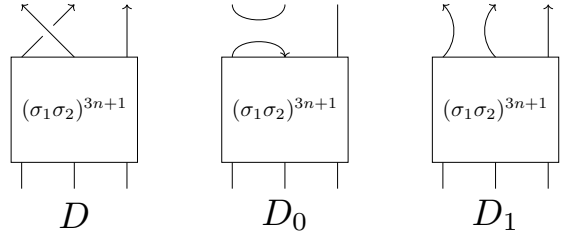
\begin{center}
\scalebox{1.4}{
\begin{tikzpicture}[scale=.5]
\node[scale=0.5] at (2,3){\usebox\torusbox};
\node[scale=.6] at (2,3) {$(\sigma_1\sigma_2)^{3n+1}$};
\node[scale=.5] at (2,4.5){\usebox\sigmaoneori};
\node at(2,1){$D$};
\node[scale=.5] at (2+4,3){\usebox\torusbox};
\node[scale=.6] at (2+4,3) {$(\sigma_1\sigma_2)^{3n+1}$};
\node[scale=.5] at (2+4,4.5){\usebox\sigmaoneZEROori};
\node at(2+4,1){$D_0$};
\node[scale=.5] at (2+8,3){\usebox\torusbox};
\node[scale=.6] at (2+8,3) {$(\sigma_1\sigma_2)^{3n+1}$};
\node[scale=.5] at (2+8,4.5){\usebox\sigmaoneONEori};
\node at(2+8,1){$D_1$};
\end{tikzpicture}
}
\end{center}

The diagram $D_0$ is a diagram of the unknot $U$ and $D_1$ is a diagram of $T(3,3n+1)$. The top $\sigma_1$ in $\Delta^{2n+1}$ is a negative crossing so we compute $c=n_-(D_0)-n_-(D)=(1+2n)-(6n+3)=-4n-2.$ Using (\ref{negativeLES}) for each $j$, and letting $\F=\Q$ or $\zp$ where $p$ is an odd prime, we get a long exact sequence
\begin{equation}\label{LESOmega3}
\xrightarrow{\delta_*} H^{i,j+1}(T(3,3n+1);\F) \longrightarrow H^{i,j}(D;\F) \longrightarrow H^{i+4n+2,j+12n+5}(U;\F) \xrightarrow{\delta_*} H^{i+1,j+1}(T(3,3n+1);\F).
\end{equation}
For $i\neq -4n-2, -4n-1$, we have $H^{i+4n+2,j+12n+5}(U;\F)=0=H^{i+4n+1,j+12n+5}(U;\F)$ for every $j$, so exactness yields $H^{i,j}(D;\F)\cong H^{i,j+1}(T(3,3n+1);\F)$ for every $j$. For $j\neq -12n-5\pm1$, the portion of the long exact sequence containing $i=-4n-2$ and $-4n-1$ splits as
\[ 0 \longrightarrow H^{-4n-2,j+1}(T(3,3n+1);\F) \longrightarrow H^{-4n-2,j}(D;\F) \longrightarrow 0 \]
\[ 0\xrightarrow{} H^{-4n-1,j+1}(T(3,3n+1);\F) \longrightarrow H^{-4n-1,j}(D;\F) \longrightarrow 0. \]
Upon examining the homology of $T(3,3n+1)$, shown here in Figure \ref{QversusZ2}, we have the following two equalities: 
\begin{align}
    H^{-4n-2,j}(D;\F)&\cong H^{-4n-2,j+1}(T(3,3n+1);\F)=0,\label{jneq-1}\\ H^{-4n-1,j}(D;\F)&\cong H^{-4n-1,j+1}(T(3,3n+1);\F)=0,\label{jneq-2}
\end{align}
when $j\neq -12n-5\pm1$.
It remains to check the portion of the long exact sequence containing $i=-4n-2, -4n-1$ in the cases $j=-12n-6,-12n-4$:
\[ 0 \longrightarrow H^{-4n-2,-12n-3}(T(3,3n+1);\F) \longrightarrow H^{-4n-2,-12n-4}(D;\F) \longrightarrow H^{0,1}(U;\F) \]
\[ \xrightarrow{\delta_*} H^{-4n-1,-12n-3}(T(3,3n+1);\F) \longrightarrow H^{-4n-1,-12n-4}(D;\F) \longrightarrow 0, \]
\[ 0 \longrightarrow H^{-4n-2,-12n-5}(T(3,3n+1);\F) \longrightarrow H^{-4n-2,-12n-6}(D;\F) \longrightarrow H^{0,-1}(U;\F) \]
\[ \xrightarrow{\delta_*} H^{-4n-1,-12n-5}(T(3,3n+1);\F) \longrightarrow H^{-4n-1,-12n-6}(D;\F) \longrightarrow 0. \]
From Figure \ref{QversusZ2}, we obtain
\[ H^{-4n-2,-12n-3}(T(3,3n+1);\F)=0=H^{-4n-2,-12n-5}(T(3,3n+1);\F), \]
\[ H^{-4n-1,-12n-3}(T(3,3n+1);\F)=\F =H^{-4n-1,-12n-5}(T(3,3n+1);\F),\]
and of course $H^{0,\pm1}(U;\F)=\F$ for any field $\F$. Thus we have exact sequences
\begin{equation}\label{For0one}
0 \longrightarrow H^{-4n-2,-12n-4}(D;\F) \longrightarrow \F \xrightarrow{\delta_*} \F \longrightarrow H^{-4n-1,-12n-4}(D;\F) \longrightarrow 0
\end{equation}
\begin{equation}\label{For0two}
0 \longrightarrow H^{-4n-2,-12n-6}(D;\F) \longrightarrow \F \xrightarrow{\delta_*} \F \longrightarrow H^{-4n-1,-12n-6}(D;\F) \longrightarrow 0.
\end{equation}
From (\ref{For0one}) and (\ref{For0two}) it follows that each of the groups $H^{-4n-1,-12n-5\pm 1}(D;\F)$, $H^{-4n-2,-12n-5\pm 1}(D;\F)$ is isomorphic to either $\F$ or 0. 
We argue that all four of them are isomorphic to $\F$. A straightforward application of Lemma \ref{leehomology} yields the dimension $\dim_{\F}(\operatorname{Lee}^{-4n-2}(\Delta^{2n+1};\F)) = 2$. 
We found in equations (\ref{jneq-1}) and (\ref{jneq-2}) that
$H^{-4n-2,j}(\Delta^{2n+1};\F) = 0$ for $j \neq -12n-5\pm 1$, 
and therefore, $\dim_\F H^{-4n-2,*}(\Delta^{2n+1};\F)\leq 2$. Since the Lee spectral sequence has $E^1$ page the $\F$-Khovanov homology and converges to Lee homology, we must also have $\dim_\F H^{-4n-2,*}(\Delta^{2n+1};\F)\geq 2$, and so it follows that $H^{-4n-2,-12n-5\pm 1}(D;\F) = \F$. Finally, the non-triviality of these two groups together with (\ref{For0one}) and (\ref{For0two}) imply that $H^{-4n-1,-12n-4}(D;\F)\cong\F$ and $H^{-4n-1,-12n-6}(D;\F)\cong\F$.
\end{proof}

\begin{cor}\label{NoOddTorsionOmega3}
If $L\in\Omega_3$, then $H(L)$ contains no $\Z_{p^r}$ torsion for $r\geq 1$, where $p$ is an odd prime.
\label{omegathree}
\end{cor}

\subsection{Even torsion in $\Omega_0,\Omega_1,\Omega_2,\Omega_3$}

In this subsection, we use Theorem \ref{GeneralTheorem2} to explicitly compute all torsion for links in $\Omega_0,\Omega_1,\Omega_2,$ and $\Omega_3$.
Benheddi \cite[Page 94]{benheddi2017khovanov} computed the reduced $\Z_2$-Khovanov homology of the torus links $T(3,q)$, and from those computations we can recover the unreduced $\Z_2$-Khovanov homology of the torus links $T(3,q)$, by using Equation \eqref{reducediso}, and letting $j=\delta+2i$. These computations encompass the closed $3$-braids in $\Omega_0, \Omega_1,$ and $\Omega_2$. We display the $\Q$-Khovanov homology and $\Z_2$-Khovanov homology of these links in the top three rows of Figure \ref{QversusZ2}. The $\Z_2$-Khovanov homology of the closure of braids in $\Omega_3$ is computed from the $\Z_2$-Khovanov homology of $T(3,3n+1)$, similarly to the proof of Theorem \ref{omega3comp}. 

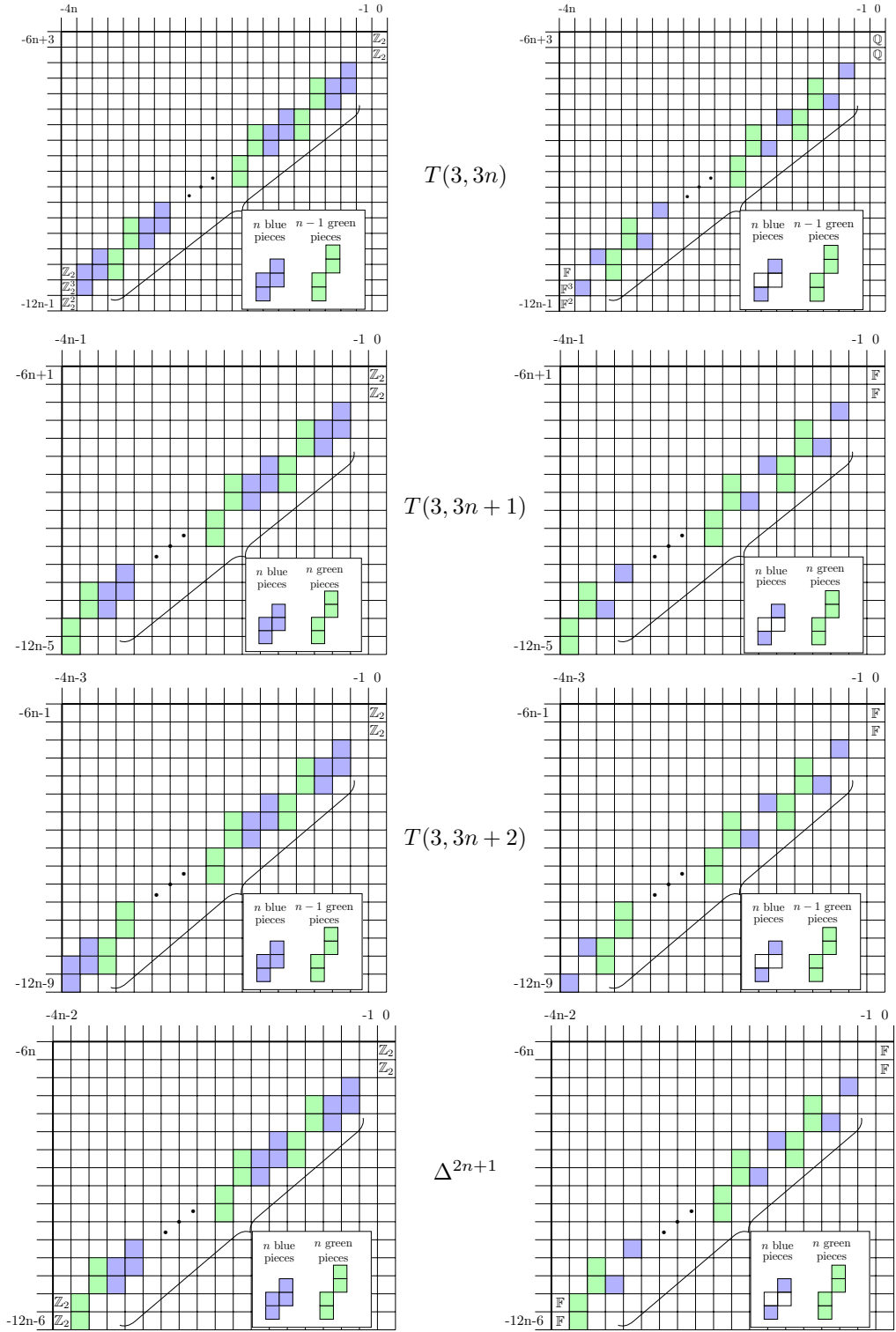
\begin{figure}
    \centering
    \begin{tikzpicture}
        \node[scale=.65*.93] at (3,2.5){\usebox\OddPowersOfDeltaModTwoColored};
        \node[scale=.65*.93] at (10.5,2.5){\usebox\OddPowersOfDeltaOverQColored};
        \node[scale=.65*.93] at (3,7.6){\usebox\TorusThreeNPlusTwoModTwoColored};
        \node[scale=.65*.93] at (10.5,7.6){\usebox\TorusThreeNPlusTwoOverQColored};
        \node[scale=.65*.93] at (3,12.7){\usebox\TorusThreeNPlusOneModTwoColored};
        \node[scale=.65*.93] at (10.5,12.7){\usebox\TorusThreeNPlusOneOverQColored};
        \node[scale=.56*.93] at (3.05,17.8){\usebox\TorusThreeNModTwoColored};
        \node[scale=.56*.93] at (10.55,17.8){\usebox\TorusThreeNOverQColored};
        \node at (7,2.5){$\Delta^{2n+1}$};
        \node at (7,7.5){$T(3,3n+2)$};
        \node at (7,12.5){$T(3,3n+1)$};
        \node at (7,17.5){$T(3,3n)$};
    \end{tikzpicture}
    \caption{For the families $T(3,3n)$, $T(3,3n+1)$, $T(3,3n+2)$, $\Delta^{2n+1}$ of links, we show in the left (resp. right) column, the Khovanov homology over $\Z_2$ (resp. $\F=\Q$ or $\Z_p$ where $p$ is an odd prime). Each colored box represents a single copy of $\Z_2$ (resp $\F$) which is killed in the Turner (resp. Lee) spectral sequence. }
    \label{QversusZ2}
    \label{oddpowerdeltatableF}
\end{figure}

\begin{thm}
For any $n \ge 0$, 
\[H(\Delta^{2n+1};\Z_2)\cong H(T(3,3n+1);\Z_2)\{-1\}\oplus H(U;\Z_2)[-4n-2]\{-12n-5\}.\]
\end{thm}
\begin{proof}
For homological gradings $-4n-1$ through $0$, the proof of this theorem is largely the same as the proof of Theorem \ref{omega3comp}. We focus on homological grading $-4n-2$. From (\ref{For0one}) and (\ref{For0two}) it follows that each of the groups \[H^{-4n-1,-12n-5\pm1}(D;\Z_2), \  H^{-4n-2,-12n-5\pm 1}(D;\Z_2)\] 
is isomorphic to either $\Z_2$ or the trivial group. We argue that each of these groups is isomorphic to $\Z_2$.
Using Lemma \ref{Bar-Natan dimension}, we find that $\dim_{\Z_2}(\BN^{-4n-2}(\Delta^{2n+1})^\prime) = 2$. Using Benheddi's calculation \cite{benheddi2017khovanov} of the $\Z_2$-Khovanov homology of $T(3,3n+1)$, shown in Figure \ref{QversusZ2}, the long exact sequence (\ref{LESOmega3}) gives \[H^{-4n-2,j}(\Delta^{2n+1};\Z_2) \cong H^{-4n-2,j+1}(T(3,3n+1);\Z_2) = 0\]
for $j \neq -12n-6, -12n-4$. Therefore, $\dim_{\Z_2} H^{-4n-2,*}(\Delta^{2n+1};\Z_2)\leq 2$. Since the Turner spectral sequence has $E^1$ page the $\Z_2$-Khovanov homology, and converges to Bar-Natan homology, 
\[\dim_{\Z_2} H^{-4n-2,*}(\Delta^{2n+1};\Z_2) \ge 2.\]
Therefore it follows that $H^{-4n-2,-12n-6}(D;\Z_2) \cong \Z_2$ and $H^{-4n-2,-12n-4}(D;\Z_2) \cong \Z_2$. Finally, the non-triviality of these two groups together with (\ref{For0one}) and (\ref{For0two}) imply that $H^{-4n-1,-12n-4}(D;\Z_2)\cong\Z_2$ and $H^{-4n-2,-12n-6}(D;\Z_2)\cong\Z_2$. 
\end{proof}

The Khovanov homology with $\Z_2$ and $\Q$ coefficients of the closure of $\Delta^{2n+1}$ is depicted in Figure \ref{QversusZ2}.
The computations of Khovanov homology with $\Q$ and $\Z_p$ coefficients for closed braids in $\Omega_0, \Omega_1, \Omega_2$, and $\Omega_3$ leads to the following application of Theorem \ref{GeneralTheorem2}.
\begin{thm}\label{TorusLinksOnly2Torsion}\label{mainthm2}
All torsion in the Khovanov homology of a closed 3-braid $L$ of type $\Omega_0,\Omega_1,\Omega_2$ or $\Omega_3$ is $\Z_2$ torsion, that is,  $H(L)\cong\Z^k\oplus \Z_2^\ell$ for some $k,\ell\geq 0$.
\end{thm}
\begin{proof}
Let $L$ be a closed braid in $\Omega_0,\Omega_1,\Omega_2$ or $\Omega_3$. Theorem \ref{toruslinksoddtorsion} and Corollary \ref{NoOddTorsionOmega3} imply $L$ contains no $\Z_{p^r}$ torsion for any odd prime $p$. Therefore $L$ satisfies condition (2) of Theorem \ref{GeneralTheorem2} on any thin region. 
Figure \ref{QversusZ2} shows that all torsion in $H(L)$ occurs in the thin ``blue'' regions, and moreover, no torsion is supported in the initial homological grading of any thin region. Thus each thin ``blue'' region satisfies conditions 1 and 3 of Theorem \ref{GeneralTheorem2}. Finally, if we look at any one of the thin ``blue" regions in Figure \ref{QversusZ2}, we see that condition (4) is satisfied in the preceding homological grading with one exception. Figure \ref{QversusZ2} shows that in $H(T(3,3n);\Z_2)$ the first blue piece does not satisfy condition (4) of Theorem \ref{GeneralTheorem2}. Recall however that these $\Z_2$ summands to the left all survive to the infinity page of the Turner and Lee spectral sequences, so for this one piece, the concerned reader can apply the stronger Theorem \ref{GeneralTheorem}. 
We conclude that all torsion in $H(L)$ is $\Z_2$ torsion.
\end{proof}
As corollaries, we obtain the integral Khovanov homology of closed $3$-braids in $\Omega_0, \Omega_1, \Omega_2$, and $\Omega_3$.
\begin{cor}\label{IntegralHomologyOfDelta2nPlusOne}
For any $n \ge 0$, $H(\Delta^{2n+1})\cong H(T(3,3n+1))\{-1\}\oplus H(U)[-4n-2]\{-12n-5\}.$
\end{cor}

\begin{cor}\label{IntegralHomologyCalculations}
The integral Khovanov homology of links in classes $\Omega_0, \Omega_1, \Omega_2$, and $\Omega_3$ are given in Figures \ref{torusthreeNoverZtable}, \ref{torusthreeNPlusOneoverZtable}, \ref{torusthreeNPlusTwooverZtable}, and \ref{oddpowerdeltatableZ}.
\end{cor}

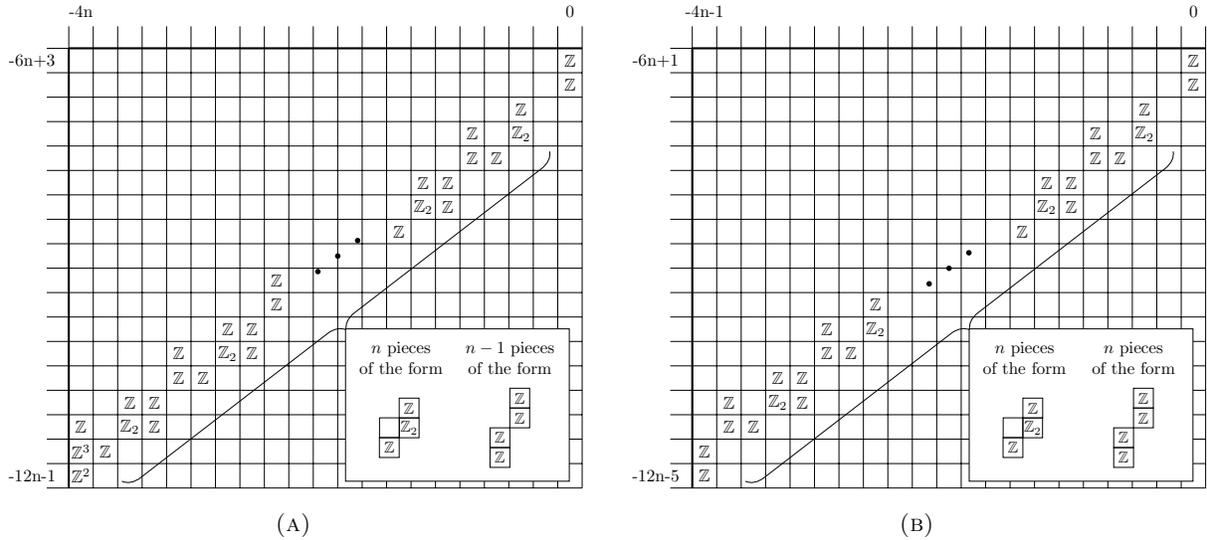
\begin{figure}[h!]
\begin{subfigure}{0.49\textwidth}
\centering
\begin{tikzpicture}
\node[scale=.65] at (0,0){\usebox\TorusThreeNOverZ};
\end{tikzpicture}
\caption{}
\label{torusthreeNoverZtable}
\end{subfigure}
\begin{subfigure}{0.5\textwidth}
\centering
\begin{tikzpicture}
\node[scale=.65] at (0,0){\usebox\TorusThreeNPlusOneOverZ};
\end{tikzpicture}
\caption{}
\label{torusthreeNPlusOneoverZtable}
\end{subfigure}
\caption{In (A) we have the integral Khovanov homology of $T(3,3n)$. In (B) we have the integral Khovanov homology of $T(3,3n+1)$.}
\end{figure}

\begin{figure}[h!]
\begin{subfigure}{0.49\textwidth}
\centering
\begin{tikzpicture}
\node[scale=.65] at (0,0){\usebox\TorusThreeNPlusTwoOverZ};
\end{tikzpicture}
\caption{}
\label{torusthreeNPlusTwooverZtable}
\end{subfigure}
\begin{subfigure}{0.5\textwidth}
\centering
\begin{tikzpicture}
\node[scale=.65] at (0,0){\usebox\DeltaTwoNPlusOneZ};
\end{tikzpicture}
\caption{}
\label{oddpowerdeltatableZ}
\end{subfigure}
\caption{A) The integral Khovanov homology of $T(3,3n+2)$. (B) The integral Khovanov homology of the braid closure of $\Delta^{2n+1}$.}
\end{figure}
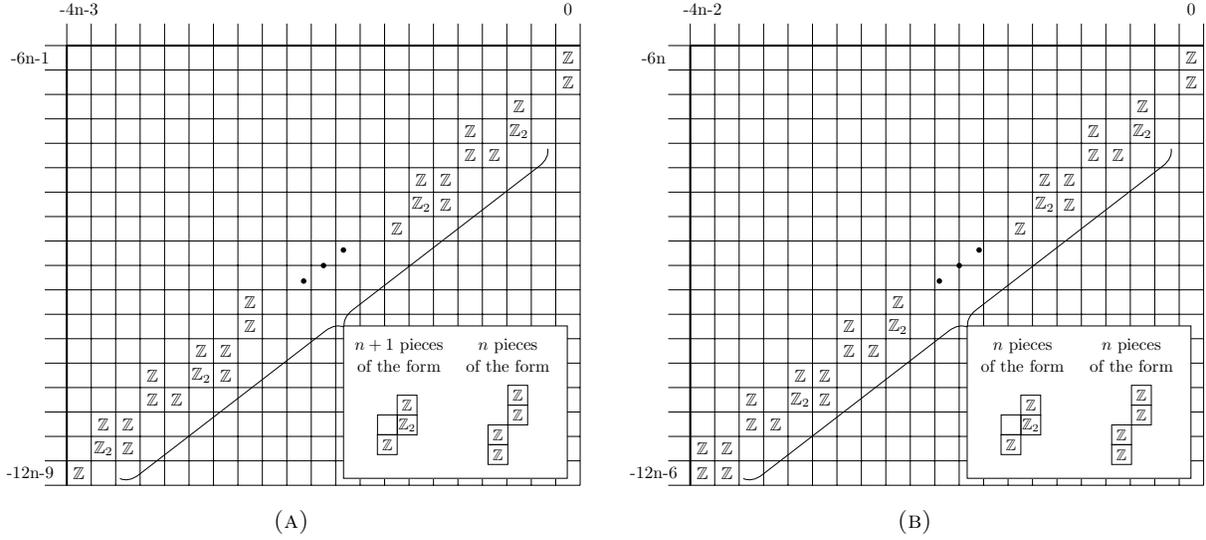

\subsection{Closed $3$-braids in $\Omega_4$, $\Omega_5$, and $\Omega_6$}
\label{limitations}

One goal of this project is to prove part (1) of Conjecture \ref{psbraid}, that closed 3-braids have only $\Z_2$ torsion in Khovanov homology. Based on Murasugi's classification shown in Theorem \ref{murasugi1}, we have confirmed this result for links in the classes $\Omega_i$ for $0\leq i\leq 3$, leaving only the classes $\Omega_4$, $\Omega_5,$ and $\Omega_6$. We now point out examples from these classes for which Theorem \ref{GeneralTheorem} is insufficient. In a future paper, we plan to use these examples as a guide to come up with a stronger version of Theorem \ref{GeneralTheorem} which can be used to deal with these remaining cases, perhaps by showing a relationship between higher order Bockstein and Turner differentials, as suggested by Shumakovitch \cite{shumakovitch2018torsion}. 

\begin{figure}[h!]
\begin{subfigure}{0.24\textwidth}
\centering
\begin{tikzpicture}
\node[scale=.65] at (0,0){\usebox\OmegaFourCounterQ};
\end{tikzpicture}
\caption{}
\label{OmegafourcounterQ}
\end{subfigure}
\begin{subfigure}{0.24\textwidth}
\centering
\begin{tikzpicture}
\node[scale=.65] at (0,0){\usebox\OmegaFourCounterZTwo};
\end{tikzpicture}
\caption{}
\label{OmegafourcounterZTwo}
\end{subfigure}
\begin{subfigure}{0.24\textwidth}
\centering
\begin{tikzpicture}
\node[scale=.65] at (0,0){\usebox\OmegaFiveCounterQ};
\end{tikzpicture}
\caption{}
\label{OmegafivecounterQ}
\end{subfigure}
\begin{subfigure}{0.24\textwidth}
\centering
\begin{tikzpicture}
\node[scale=.65] at (0,0){\usebox\OmegaFiveCounterZTwo};
\end{tikzpicture}
\caption{}
\label{OmegafivecounterZTwo}
\end{subfigure}
\caption{In (A) we have the rational Khovanov homology of the closure of the 3-braid $\Delta^2\sigma_1^{-5}\in\Omega_4$. In (B) we have the mod 2 Khovanov homology of the closure of the 3-braid $\Delta^2\sigma_1^{-5}\in\Omega_4$. Theorem \ref{GeneralTheorem} can not be applied here due to the homology being supported on 3 diagonals in homological grading 0. In (C) we have the rational Khovanov homology of the closure of the 3-braid $\Delta^2\sigma_2\in\Omega_5$. In (D) we have mod 2 Khovanov homology of the closure of the 3-braid $\Delta^2\sigma_2\in\Omega_5$. Again, Theorem \ref{GeneralTheorem} can not be applied due to Khovanov homology being supported on 3 diagonals in homological grading -4.} 
\end{figure}
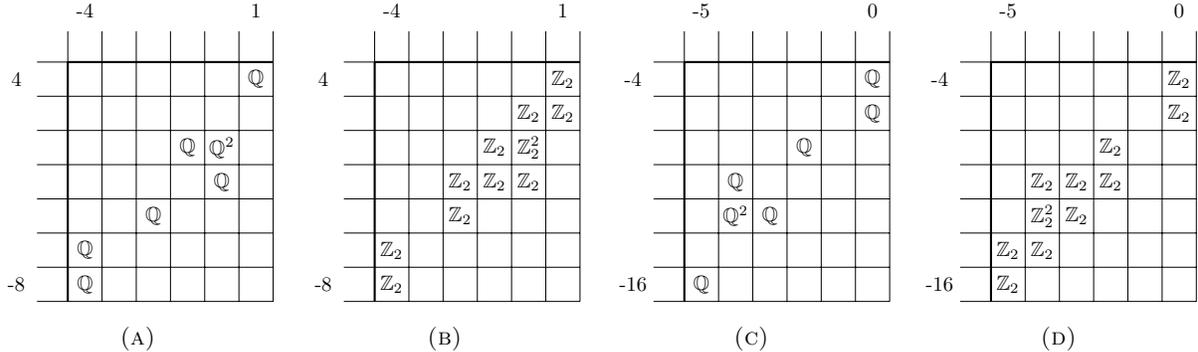

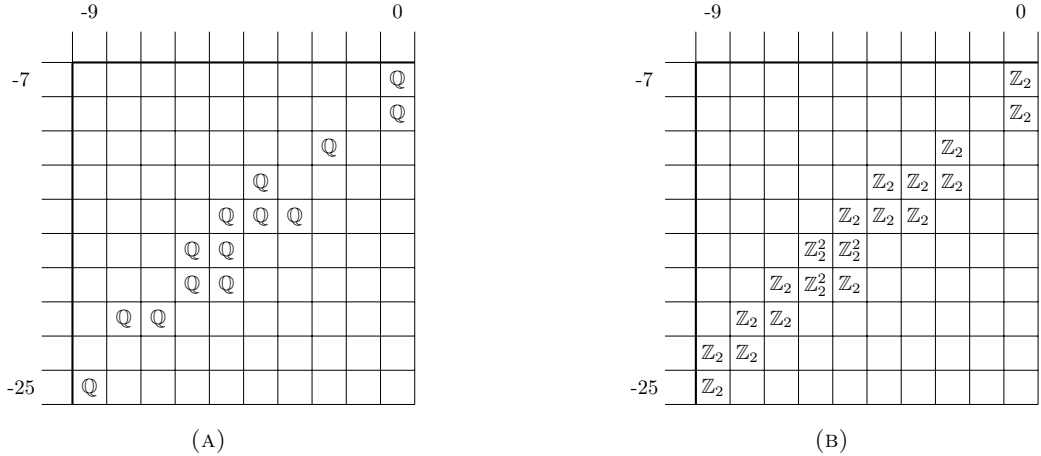
\begin{figure}[h!]
\begin{subfigure}{0.49\textwidth}
\centering
\begin{tikzpicture}
\node[scale=.65] at (0,0){\usebox\OmegaSixCounterQ};
\end{tikzpicture}
\caption{}
\label{OmegasixcounterQ}
\end{subfigure}
\begin{subfigure}{0.5\textwidth}
\centering
\begin{tikzpicture}
\node[scale=.65] at (0,0){\usebox\OmegaSixCounterZTwo};
\end{tikzpicture}
\caption{}
\label{OmegasixcounterZTwo}
\end{subfigure}
\caption{In (A) we have the rational Khovanov homology of the closure of the 3-braid $\Delta^4\sigma_1^{-2}\sigma_2\sigma_1^{-1}\in\Omega_6$. In (B) we have the mod 2 Khovanov homology of the closure of the 3-braid $\Delta^4\sigma_1^{-2}\sigma_2\sigma_1^{-1}\in\Omega_6$. Theorem \ref{GeneralTheorem} cannot be applied in this case due to the homology being supported on 3 diagonals in homological grading -5. }
\end{figure}

\FloatBarrier

\bibliographystyle{abbrv}
\bibliography{bib}


\end{document}